\documentclass[12pt,reqno]{amsart}

\usepackage[dvips]{graphicx}
\usepackage{graphics,color}
\usepackage{a4wide}
\usepackage{amssymb,amsmath,amstext,amsthm,amsfonts}
\usepackage{euscript, pxfonts}

\title[Adapted random perturbations N.U.E. maps]{Adapted random perturbations for non-uniformly expanding maps}

\date{\today}

\author{Vitor Araujo}
\address[V.A.]{Instituto de Matem\'atica,
Universidade Federal da Bahia\\
Av. Adhemar de Barros, S/N , Ondina,
40170-110 - Salvador-BA-Brazil}
\email{vitor.d.araujo@ufba.br}

\author{Maria Jose Pacifico}
\address[M.J.]{Instituto de Matem\'atica,
Universidade Federal do Rio de Janeiro\\
C. P. 68.530, 21.945-970 Rio de Janeiro}
\email{pacifico@im.ufrj.br}

\author{Mariana Pinheiro}
\address[M.P.]{Departamento de Ci\^encias Exatas e da Terra, Universidade Estadual de Santa Cruz, Ilh\'eus, Brazil.}
\email{mariana@im.ufrj.br}

\thanks{Authors were partially supported by CAPES, CNPq, FAPERJ and
  PRONEX/DS from Brazil.}

\subjclass{Primary: 37D25; Secondary: 37H99.}
 \keywords{stochastic stability,
  additive perturbation, non-uniform expansion, slow
  recurrence, adapted random perturbation, stationary
  measures, physical measures, first hyperbolic time map.}

\newtheorem{maintheorem}{Theorem}

\newtheorem{theorem}{Theorem }[section]
\newtheorem{proposition}[theorem]{Proposition}
\newtheorem{lemma}[theorem]{Lemma}
\newtheorem{corollary}[theorem]{Corollary}

\newtheorem{remark}[theorem]{Remark}

\newtheorem{definition}[theorem]{Definition}

\newtheorem{conjecture}{Conjecture}

\newcommand{\be} {\beta}

\newcommand{\ep} {\epsilon}

\newcommand{\vfi}{\varphi}

\newcommand{\Z}{\mathbb{Z}}
\newcommand{\N}{\mathbb{N}}
\newcommand{\sS}{\mathbb{S}}
\newcommand{\R}{\mathbb{R}}

\newcommand{\nat}{\mathbb{N}}

\newcommand{\T}{\mathbb{T}}

\newcommand{\un}{\underline}

\newcommand{\SB}{{\mathcal B}}
\newcommand{\SC}{{\mathcal C}}
\newcommand{\SD}{{\mathcal D}}
\newcommand{\SE}{{\mathcal E}}
\newcommand{\SF}{{\mathcal F}}

\newcommand{\supp}{\operatorname{supp}}

\begin{document}

\begin{abstract}
  We obtain stochastic stability of $C^2$ non-uniformly
  expanding one-dimensional endomorphisms, requiring only
  that the first hyperbolic time map be $L^{p}$-integrable
  for $p>3$. We show that, under this condition (which
  depends only on the unperturbed dynamics), we can
  construct a random perturbation that preserves the
  original hyperbolic times of the unperturbed map and,
  therefore, to obtain non-uniform expansion for random
  orbits.  This ensures that the first hyperbolic time map
  is uniformly integrable for all small enough noise levels,
  which is known to imply stochastic stability. The method
  enables us to obtain stochastic stability for a class of
  maps with infinitely many critical points. For higher
  dimensional endomorphisms, a similar result is obtained,
  but under stronger assumptions.
\end{abstract}

\maketitle
\tableofcontents

\section{Introduction}
\label{sec:introduction}

The main goal of Dynamical systems theory is the description
of the typical behaviour of orbits as time goes to infinity,
and to understand how this behaviour changes under small
perturbations of the system.

Given
a map $f$ from a manifold $M$ into itself, a central
concept is that of \emph{physical measure}, a
$f$-invariant probability measure $\mu$ whose
\emph{ergodic basin}
\begin{align}\label{eq:ergbasin}
  B(\mu)=\big\{x\in M: n^{-1}\sum_{j=0}^{n-1}\phi(f^j(x))
  \xrightarrow[n\to+\infty]{} \int\phi\,d\mu\mbox{ for all
    continuous } \phi: M\to\R\big\}
\end{align}
has positive \emph{volume} or \emph{Lebesgue measure}, which
we write $\lambda$ and take as the measure associated with any
non-vanishing volume form on $M$.

The stability of physical measures under small variations of
the map %and under random perturbations of the transformation
allows for small errors along orbits not to disturb too much
the long term behavior, as measured by % the most basic
% statistical data provided by
asymptotic time averages of continuous functions along
orbits.  When considering practical systems we cannot avoid
external noise, so every realistic mathematical model should
exhibit these stability features to be able to deal with
uncertainty about parameter values, observed initial states
and even the specific mathematical formulation of the model
itself.

We investigate, under the probabilistic point of view, which
asymptotic properties of a dynamical system are preserved
under random perturbation.

Random perturbations and their features were first studied
in 1945 by Ulam and von Neumann, in \cite{UvN}.  The focus
of this work are non-uniformly expanding transformations
which were introduced by Alves-Bonatti-Viana in
\cite{ABV00}, and whose ergodic properties are now well
established; see for instance
\cite{Ze03,AA03,AlV13,AlLuPi03}. Here we show that the
asymptotic behavior of these transformations is preserved
when randomly perturbed in an adapted way to their first
times of expansion, under a condition: that the first time
of expansion is $L^p$-integrable with respect to Lebesgue
measure; see next sections for precise definitions and
statements.

The interest in this kind of stochastic stability condition
lies in the fact that \emph{known conditions of stochastic
  stability for non-uniformly expanding maps are expressed
  as conditions on the random perturbations of the map} and
not solely on the original unperturbed dynamics. We mention
the joint works with Alves \cite{AA03} and Vasquez
\cite{alarva}, and also the recent work by Alves and
Vilarinho \cite{AlV13}.

The uniformly hyperbolic case, studied by Kifer in
\cite{Ki88,Ki86} (among others), is much simpler: uniformly
hyperbolic systems are stochastically stable under a broad
range of random perturbations without further
conditions. Other cases with the same features, which we may
say are ``almost uniformly hyperbolic systems'', where
studied in joint works with Tahzibi, in \cite{ArTah,ArTah2}.

Here, we present a sufficient condition for stochastic
stability of non-uniformly expanding transformations that
relies only on the dynamics of the unperturbed map, for a
simple type of random perturbation that is adapted to the
dynamics. This allows us to treat some exceptional cases.

Recently Shen \cite{Shn13} obtained stochastic stability for
unimodal transformations under very weak assumptions, but
does not cover the case of transformations with infinitely
many critical points; and Shen together with van Strien in
\cite{ShvS13} obtained strong stochastic stability for the
Manneville-Pomeaux family of intermittent maps, answering
questions raised in \cite{ArTah}.

Our method allows us to obtain stochastic stability for
non-uniformly expanding endomorphisms having slow recurrence
to the critical set, encompassing the family of
infinite-modal applications presented in \cite{PRV98}.  We
also obtain stochastic stability (in the weak$^*$ sense, see
precise statements in the next sections) for intermittent
maps but in a restricted interval of parameter values; see
Section~\ref{sec:intermitent-maps}.

%%%%%%%%%%%%%%%%%%%%%%%%%%%%%%%%%%%%%%%%%%%%%%%%%

\subsection{Setting and statement of results}
\label{sec:statement-results}

We consider $M$ to be a $n$-torus, $\mathbb{T}^n=
(\mathbb{S}^{1})^n$, for some $n\ge1$ and $\lambda$ a
normalized volume form in $\mathbb{T}^n$, which we call
Lebesgue measure. This can be identified with the
restriction of Lebesgue measure on $\R^n$ to the unit cube.

We write $d$ for the standard distance function on $\T^n$ in
what follows and $\|\cdot\|$ for the standard Euclidean norm
on $\R^n$ which can be identified with the tangent space at
any point of $\T^n$.

We let $f:\T^n\to\T^n$ be a local $C^2$
diffeomorphism outside a \emph{non-degenerate critical set}
$\SC$, that is, $\SC=\{x\in M:\det Df(x)=0\}$ and $f$
behaves as the power distance to $\SC$: there are constants
$B>1$ and $\beta>0$ satisfying
\begin{itemize}
\item[(S1)]
  $\frac{1}{B}\cdot d(x,\SC)^{\beta}\leq
  \frac{\|Df(x)\cdot v\|}{\|v\|}
  \leq
  B \cdot d(x,\SC)^{-\beta}, \vec0\neq v\in T_x M$;
\item[(S2)]$\big|\log\| Df(x)^{-1}\|-\log\|Df(y)^{-1}\|\big|
 \leq
  B\frac{ d(x,y)}{d(x,\SC)^{\beta}}$;
 \item[(S3)]
$\big|\log|\det Df(x)^{-1}|-
\log|\det Df(y)^{-1}|\:\big|\leq
B\frac{d(x,y)}{d(x,\SC)^{\be}}$;
\end{itemize}
for all $x,y\in M\setminus\SC$ with
$\d(x,y)<\frac12\d(x,\SC)$.

We say that $f$ is \emph{non-uniformly expanding} if there
is a constant $c>0$ such that:
\begin{align}\label{nue}
\limsup_{n\to+\infty}\frac1n\sum_{j=0}^{n-1}
\log\| Df(f^{j}(x))^{-1}\|
\leq-c<0 \quad\text{for}\quad \lambda-\text{a.e.  }\,x\in M.
\end{align}
We need to control the recurrence to the critical set in
order to obtain nice ergodic properties. We say that $f$ has a
\emph{slow recurrence to critical set} if, for any given $\gamma>0$
there exists $\delta>0$ such that
\begin{align}\label{eq:slowrec}
  \limsup_{n\to+\infty}\frac1n
  \sum_{j=0}^{n-1}-\log d_{\delta}(f^{j}(x),\SC)
  \leq \gamma, \quad\text{for}\quad \lambda-\text{a.e.  }\,x\in M,
\end{align}
where $d_\delta$ is the \emph{$\delta$-truncated distance to
$\SC$}, defined as $d_{\delta}(x,\SC)=d(x,\SC)$ if
$d(x,\SC)<\delta$ and $d_{\delta}(x,\SC)=1$ otherwise.

We recall the concept of physical measure.
For any $f$-invariant probability measure $\mu$ we write
$B(\mu)$ for the \emph{basin} of $\mu$ as in~(\ref{eq:ergbasin}).
We say that a $f$-invariant measure $\mu$ is \emph{physical}
if its basin $B(\mu)$ has positive Lebesgue measure:
$\lambda(B(\mu))>0$.  

Roughly speaking, physical measures are those that can be
``seen'' by calculating the time average of the values of a
continuous observable along the orbits the points on a
subset with positive Lebesgue measure.  Clearly Birkhoff's
Ergodic Theorem ensures that $\mu(B(\mu))=1$ whenever $\mu$
is $f$-ergodic. We note that every $f$-invariant ergodic
probability measure $\mu$ which is also absolutely
continuous with respect to Lebesgue measure,
i.e. $\mu\ll\lambda$, is a physical measure.

The previous conditions on $f$ ensure that Lebesgue almost
all points behave according to some physical measure.

\begin{theorem}[Theorem C, \cite{ABV00}]
  \label{thm:abv}
  Let $f$ be $C^2$ diffeomorphism away from a non-degenerate
  critical set, which is also a non-uniformly expanding map
  whose orbits have slow recurrence to the critical
  set. Then there is a finite number of $f$-invariant
  absolutely continuous ergodic (physical) measures
  $\mu_{1}, \ldots,\mu_{p}$ whose basins cover a set of full
  measure, that is
  $\lambda\big(M\setminus(B(\mu_1)\cup\dots\cup
  B(\mu_p)\big)=0$.  Moreover, each $f$-invariant absolutely
  continuous probability measure $\mu$ can be written as a
  convex linear combination the physical measures: there are
  $\alpha_1=\alpha_1(\mu),\dots,\alpha_p=\alpha_p(\mu)\ge0$
  such that $\sum\alpha_i=1$ and $\mu=\sum\alpha_i\mu_i$.
\end{theorem}

\begin{remark}\label{rmk:pinheiro}
  Pinheiro~\cite{Pinheiro05} showed that the same
  conclusions of Theorem \ref{thm:abv} can be obtained by
  replacing the of non-uniform expansion condition
  (\ref{nue}) by the weaker condition
  \begin{eqnarray}\label{liminf}
    \liminf_{n\to+\infty}\frac1n\sum_{j=0}^{n-1}
    \log\| Df(f^{j}(x))^{-1}\|\leq
    -c<0,  \quad\text{for}\quad \lambda-\text{a.e.  }\,x\in M.
  \end{eqnarray} 
  The proof of this fact involves showing that
  (\ref{liminf}) implies (\ref{nue}). Therefore, all the
  arguments used in this paper remain valid in the more
  general setting of condition \eqref{liminf} replacing
  condition~(\ref{nue}).
\end{remark}

\subsection{Random perturbations and stochastic stability}
\label{sec:random-perturb-stoch}

We let $B=B(0,1)$ denote the unitary ball centered at the
origin $0$ in $\R^n$, set $X=\overline{B}$ and $\mathcal{F}=\{
f_{t}:M\to M; t\in X\}$ a parametrized family of maps. We
write $f_t(x)=f(t,x),(t,x)\in X\times M$ and assume in what
follows that $f_{0}=f$ is a map in the setting of
Theorem~\ref{thm:abv}.

We consider also the family of probability measures
$(\theta_{\epsilon})_{\epsilon>0}$ in $X$ given by the
normalized restriction of Lebesgue measure to the
$\epsilon$-ball $B(0,\epsilon)$ centered at $0$ in $\R^n$,
as follows
\begin{align}\label{eq:theta_ep}
  \theta_{\epsilon}=\frac{\lambda\mid_{B(0,\epsilon)}}{\lambda(B(0,\epsilon))}.
\end{align}
This family is such that $\supp(\theta_{\epsilon})_{\epsilon>0}$ is a nested
family of compact and convex sets satisfying
$\supp(\theta_{\epsilon})\xrightarrow[\epsilon\rightarrow0]{}0$.
Setting $\Omega=X^\N$ the space of sequences in $X$,
the random iteration of $\SF$ is defined by
\begin{align*}
  f^{n}_{\underline{t}}(x)=\left(f_{t_{n}}\circ
  f_{t_{n-1}}\circ\ldots\circ f_{t_{1}}\right)(x), \quad
\underline{t}=(t_{1},t_{2},\ldots)\in\Omega, x\in M.
\end{align*}
To define the notions of stationary and ergodic measure we
consider the skew-product
\begin{align}\label{continuidade dupla}
  \begin{array}{cccc}F:&\Omega\times M&\rightarrow& \Omega\times M\\
    &(\underline{t},x)&\mapsto&(\sigma(\underline{t}),f_{t_{1}}(x))
  \end{array}
\end{align}
where $\sigma:\Omega\rightarrow \Omega$ is a standard left
shift, and the infinite product measure $\theta_\epsilon^\N$
on $\Omega$, which is a probability measure on the Borel
subsets of $\Omega$ in the standard product topology.

From now on, for each $\epsilon>0$, we refer to
$(f_{\underline{t}},\theta_{\epsilon}^{\mathbb{N}})$ as a
\emph{random dynamical system} with noise of level $\epsilon$.

\begin{definition} [Stationary measure]
\label{eq:defstationary}
  A measure $\mu^\epsilon$ is a {\it stationary measure} for the
  random system
  $(f_{\underline{t}},\theta_{\epsilon}^{\mathbb{N}})$ if
  \begin{align*}
    \int\phi\,d\mu^\epsilon =
    \int\int\phi(f_{t}(x))\,d\mu^\epsilon(x)\,d\theta_{\epsilon}(t),
    \quad\text{for all}\quad \phi\in C^0(M,\R).
  \end{align*}
\end{definition}

\begin{remark}
  \label{re:accinvariant}
  If $(\mu^\epsilon)_{\epsilon>0}$ is a family of stationary
  measures having $\mu^0$ as a weak$^*$ accumulation point
  when $\ep\searrow0$, then from~(\ref{eq:defstationary})
  and the convergence of $\supp(\theta_\epsilon)$ to $\{0\}$
  it follows that $\mu^0$ must be invariant by $f=f_{0}$;
  see e.g. \cite{AA03}.
\end{remark}

We say that $\mu$ is a \emph{stationary measure} if the
measure $\theta_{\epsilon}^{\mathbb{N}}\times\mu$ is
$F$-invariant.  Moreover, we say that \emph{a stationary
  measure $\mu$ is ergodic} if
$\theta_{\epsilon}^{\mathbb{N}}\times\mu$ is $F$-ergodic. 

% We have the following form of the Ergodic Theorem in this
% setting.

% \begin{theorem}[Birkhoff Ergodic Theorem]\label{birkhoff}
%   Assume that $\mu$ is a ergodic and stationary probability
%   measure. Given $\varphi:M\to\R$ a
%   continuous function, consider $\psi:\Omega\times
%   M\to\R$ given by $\psi=\varphi\circ\pi$
%   where $\pi:\Omega\times M\to M$ is the natural
%   projection on the second coordinate. Then
% $
%     \lim_{n\to+\infty}\frac1n
%     \sum_{j=0}^{n-1}\psi(F^{j}(\underline{t},x))
%     =
%     \int\psi\,d(\theta_{\epsilon}^{\mathbb{N}}\times\mu)
% $
% for
% $\theta_{\epsilon}^{\mathbb{N}}\times\mu$-a.e
% $(\underline{t},x)\in\Omega\times M$,
% which is equivalent to
% $\lim_{n\to+\infty}\frac1n\sum_{j=0}^{n-1}
% \varphi(f_{\underline{t}}^{j}(x))=\int\varphi\, d\mu$
% for $\theta_{\epsilon}^{\mathbb{N}}\times\mu$-a.e.
% $(\underline{t},x)\in\Omega\times M$.
% \end{theorem}

\begin{definition}
  We say that $f$, in the setting of
  Theorem~\ref{thm:abv}, is {\em stochastically stable under
  the random perturbation given by} $(f_{\un
    t},\theta_\epsilon^\N)_{\epsilon>0}$ if, for all
  weak$^*$ accumulation points $\mu^0$ of families
  $(\mu^\epsilon)_{\epsilon>0}$ of stationary measures for
  the random dynamical system $(f_{\un
    t},\theta_\epsilon^\N)$ when $\epsilon\searrow 0$, we
  have that $\mu^0$ belongs to the closed convex hull of
  $\{\mu_{1},\ldots,\mu_p\}$. That is, for all such weak$^*$
  accumulation points $\mu^0$ there are
  $\alpha_1=\alpha_1(\mu^0),\dots,\alpha_p=\alpha_p(\mu^0)\ge0$
  such that $\sum\alpha_i=1$ and $\mu^0=\sum\alpha_i\mu_i$.
\end{definition}

In this work we consider additive perturbations given by
families of maps with the following form
\begin{eqnarray}\label{pertaditiva}
  f_{t}(x)=f(x)+t\zeta(x)
\end{eqnarray}
where $\zeta:M\rightarrow\mathbb{R}^{+}$ is Borel measurable
and locally constant at $\lambda$-almost every point.

\begin{remark}\label{rmk:Dft_Df} 
  For such additive perturbations we have $Df_{t}(x)=Df(x)$
  for all $t\in\Omega$ and $\lambda$-a.e. $x\in M$.
\end{remark}

%%%%%%%%%%%%%%%%%%%%%%%%%%%%%%%%%

\subsection{First hyperbolic time map and adapted random
  perturbations}
\label{sec:hyperb-times-random}

The following is the fundamental concept in this work.

\begin{definition}[Hyperbolic time]
\label{def:hyptimes}
Given $\sigma<1$ and $\delta>0$, we say that $n$ is a
$(\sigma,\delta)$-hyperbolic time for $x\in
M$ if
\begin{align*}
  \prod_{j=n-k}^{n-1}\| Df(f^{j}(x))^{-1}\| \leq \sigma^{k}
  \quad\mbox{and}\quad d_{\delta}(f^{n-k}(x),\SC)\geq
  \sigma^{bk}\quad\text{for all}\quad 1\leq k\leq n,
\end{align*}
where $b=\min\{1/2, 1/2\beta\}$ and $\beta$ is the constant
given in the non-degenerate conditions (S1)-(S2).
\end{definition}

% Note that in our context of perturbation, the first
% inequality in the definition above is
% $\prod_{j=n-k}^{n-1}\p
% Df(f_{\underline{t}}^{j}(x))^{-1}\p\leq\sigma^{k}$ for
% all $0\leq k\leq n$.
The notion of hyperbolic times was defined in \cite{ABV00}.
To explain our Main Theorem we cite the following technical
result.

\begin{lemma}[Lemma 5.4 in \cite{ABV00}]\label{le:infhyptimes}
  Let $f$ be a $C^2$ local diffeomorphism away from a
  non-degenerate critical set, which satisfies the
  non-uniform expansion condition $(\ref{nue})$ with
  $c=3\log\sigma$ for some $0<\sigma<1$ and also the slow
  recurrence condition~(\ref{eq:slowrec}). 

  Then there exist $\theta_{0},\delta>0$ depending on
  $\sigma$ and $f$, such that for $\lambda$-a.e. $x$ and
  each big enough $N\geq 1$, there are
  $(\sigma,\delta)$-hyperbolic times $1\leq n_{1}<
  \dots<n_{l}\leq N$ for $x$ with
  $l\geq\theta_{0}N$. Moreover, the hyperbolic times $n_{i}$
  satisfy
\begin{align}\label{eq:hip-times-prop}
\sum_{j=n_{i}-k}^{n_{i}-1}\log d_{\delta}(f^{j}(x),\SC)
\geq bk\log\sigma,
\quad\text{for all}\quad 0\leq k\leq n_{i}, 1\leq i\leq l.
\end{align}
\end{lemma}

\begin{remark}\label{rmk:infhyptimes}
  Let $\mathcal{G}$ the set of points $x\in M$ that have no
  hyperbolic time. Then $\lambda(\mathcal{G})=0$ after Lemma
  $\ref{le:infhyptimes}$. Thus, if $x$ has only finitely
  many hyperbolic times, then some iterate of $x$ belongs to
  $\mathcal{G}$. Hence, the subset of points with finitely
  many hyperbolic times is contained in
  $\cup_{j\ge0}f^{-j}(\mathcal G)$. Moreover,
  $\lambda(f^{-j}(\mathcal G))=0$ because $f$ is a local
  diffeomorphism away from a critical/singular set with zero
  $\lambda$-measure.  Therefore, $\lambda$-a.e. $x\in M$ has
  infinitely many hyperbolic times.
\end{remark}

Hence, in our setting we have that Lebesgue almost every
point has infinitely many $(\sigma,\delta)$-hyperbolic
times. Thus we may define the map $h:M\to\Z^{+}$ such that
for $\lambda$-a.e. point $x$ the positive integer $h(x)$ is
the first hyperbolic time of $x$. We say $h$ is the {\it
  first~hyperbolic~time~map}.

In our main theorem, we will see that is possible to
randomly perturb a non-uniformly expanding map so that
almost all randomly perturbed orbits have infinitely many
hyperbolic times but also the same hyperbolic times as
the non-perturbed map. We start with a one-dimensional version.

\begin{maintheorem}\label{principal}
  Let $f:\T^1\to\T^1$ be a  $C^2$ diffeomorphism away from a
  non-degenerate critical set, which is also a non-uniformly
  expanding map whose orbits have slow recurrence to the
  critical set.  
  Let us assume that $f$ has a dense orbit and that the
  first hyperbolic time map is $L^p$-integrable for some
  $p>3$, that is, $\int h(x)^p\,d\lambda(x)<\infty$. 

  Then $f$ is stochastically stable for a family of adapted
  random perturbations.
  More precisely, there exists $\zeta:\T^1\rightarrow
  \mathbb{R}^{+}$ mensurable and locally constant such that
  the family $(\ref{pertaditiva})$ generates a family of
  random perturbations $(f_{\un
    t},\theta_\epsilon^\N)_{\epsilon>0}$ for which $f$ is
  stochastically stable.
\end{maintheorem}

The same proof gives the following result for endomorphisms
of compact manifolds in higher dimension, with a technical
assumption on the rate of decay of the measure of sets of
points with first hyperbolic time.

\begin{maintheorem}\label{mthm:higher}
  Let $f:\T^n\to\T^n$ be a  $C^2$ diffeomorphism away from a
  non-degenerate critical set, which is also a non-uniformly
  expanding map whose orbits have slow recurrence to the
  critical set, where $n>1$.  
  If the first hyperbolic time map satisfies
  \begin{align}\label{eq:summa}
    \sum_{n\geq
      1}\sum_{j=0}^{n-1}\lambda(f^{j}(h^{-1}(n)))<\infty,
  \end{align}
  then $f$ is stochastically stable for a family of adapted
  random perturbations given by (\ref{pertaditiva}).
\end{maintheorem}

\subsection{Comments and organization of the text}
\label{sec:organization-text}

The method of proof relies on showing that the random
adapted perturbation preserve hyperbolic times in such a way
that the first hyperbolic time map of the random system is
the same as the first hyperbolic time map of the original
system. In this way, we can use the main result of
\cite{AA03} to prove (weak$*$) stochastic stability.

This construction of the adapted random perturbation depends
on an assumption of integrability of the first hyperbolic
time map for one-dimensional non-uniformly expanding
maps. For higher dimensional maps,
condition~(\ref{eq:summa}) is needed and apparently much
difficult to check.

\begin{conjecture}
  \label{sec:higherdim}
  A non-uniformly expanding map having a sufficiently
  fast rate of decay of correlations satisfies the
  summability condition~(\ref{eq:summa}).
\end{conjecture}

We presented the results using a uniform measure for
$\theta_\epsilon$ but many simple generalizations are
possible assuming only that $\theta_\epsilon\ll\lambda$ and
$\supp(\theta_\epsilon)\xrightarrow[\epsilon\searrow0]{}\{0\}$.

We also avoided technical complexities by considering only
maps on tori, on which it is clear how to make additive
perturbations in the form~(\ref{pertaditiva}). However, it
is possible (although technically more involved) to make
similar perturbations in any compact manifold, arguing along
the lines of \cite[Example 2]{vdaraujo2000}. We focus on
additive perturbations on parellelizable manifolds to
present the ideas in a simple form.

\subsection*{Acknowledgments}
This is M.P. PhD thesis prepared at the Federal University
of Rio de Janeiro, Brazil. All authors are indebted to the
research facilities provided by the Mathematics Institute at
this University.

%%%%%%%%%%%%%%%%%%%%%%%%%%%%%%%%%%%%%%%%%%%%%%%%%%%%%%%%

\section{Examples of Application}
\label{sec:examples-application}

Theorem~\ref{principal} ensures stochastic stability for
any non-uniformly expanding map that has slow recurrence to
the critical set with the first hyperbolic function in
$L^{p}$ for $p>3$. We present natural conditions on
the speed of expansion that imply this integrability
condition and use this to obtain examples where our
results apply.

We note that, from  slow recurrence to the critical set
and non-uniform expansion, Lemma~\ref{le:infhyptimes}
ensures that for $c=-\log\sigma>0$ and small
$\gamma,\delta>0$ the following values are well
defined $\lambda$-a.e.
\begin{align*}
  \SD(x)
  &=
  \min\left\{k\ge1: \frac1n\sum_{j=0}^{n-1} -\log
  d_{\delta}(f^{j}(x),\SC)\leq\gamma\quad\text{for
    all}\quad n\ge k\right\};\quad\text{and}
\\
\SE(x)
&=\min\left\{k\ge1: \frac1n\sum_{j=0}^{n-1}\log
  \|Df(f^j(x))^{-1}\|\ge \frac{c}3\quad\text{for
    all}\quad n\ge k\right\}.
\end{align*}
We combine these two estimates in the set
\begin{align*}
  \Gamma_n=\{x\in M:\SD(x)>n\quad\text{and}\quad\SE(x)>n\}.
\end{align*}
We now observe that, trivially from the definitions,
every point in $\Gamma_n$ has a first
$(\sigma,\delta)$-hyperbolic time of at most $n$, thus
\begin{align*}
  h^{-1}(\{n\})\subset h^{-1}(\{1,2,\dots,n\})\subset\Gamma_n.
\end{align*}

\begin{remark}
  \label{rmk:Gamma_tail}
  If for some constant $\kappa>0$ and $q>4$ we have
  $\lambda(\Gamma_n)\le\kappa n^{-q}$ for all
  sufficiently large $n$, then $h\in L^p(\lambda)$ for
  some $p>3$, since for all small enough $\epsilon>0$
  we have $\sum_{n>m}
  n^{q-1-\epsilon}\lambda(h^{-1}(\{n\}))
  \le
  \kappa\sum_{n>m}n^{-1-\epsilon}<\infty$ for some $m>1$.
\end{remark}

%%%%%%%%%%%%%%%%%%%%%%%%%%%%%%%%%%%%%%%%%%%%%

\subsection{Non-uniformly expanding maps with infinitely
  many critical points}
\label{sec:non-uniformly-expand}

We now present the main motivating example of application of
Theorem~\ref{principal}: maps with infinite critical points.
We consider the family $f_t:\sS^1\to\sS^1$ from
the work of Pacifico-Rovella-Viana \cite{PRV98}. This
family is obtained from the map
$\hat{f}:[-\epsilon_1,\epsilon_1]\to[-1,1]$ given by
\begin{equation}
\label{e2.3}
\hat{f}(z)=\left\{
\begin{array}{ll}
az^{\alpha}\sin(\beta\log(1/z)) & \mbox{   if   }z>0\\
-a|z|^{\alpha}\sin(\beta\log(1/|z|)) & \mbox{   if   }z<0,
\end{array}
\right.
\end{equation}
where $a>0$, $0<\alpha<1, \, \beta>0$ and $\ep_1>0$, see
Figure~\ref{Fig4}.

\begin{figure}[h!] \centering
\includegraphics[height=5.5cm]{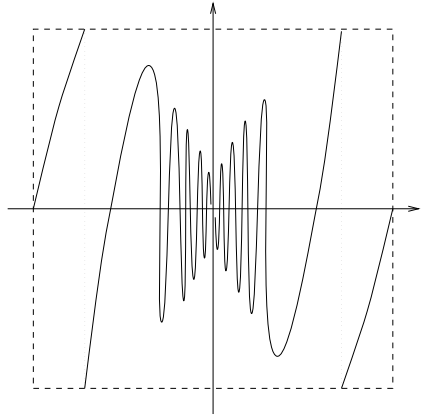}
\caption{\label{Fig4} Graph of the circle map $f$.}
\end{figure}
Maps $\hat{f}$ as above  have infinitely many critical points,
of the form
\begin{equation}
\label{e2.4}
x_{k}=\hat{x}\exp(-k\pi/\beta)
\mbox{    and    }x_{-k}=-x_{k}
\mbox{    for each large   }k>0
\end{equation}
where
$\hat{x}=\exp\big(-\frac1{\beta}\tan^{-1}\frac{\beta}{\alpha}\big)>0$
is independent of $k$.  Let $k_0\ge 1$ be the smallest
integer such that $x_k$ is defined for all $|k|\ge k_0$, and
$x_{k_0}$ is a local minimum.

We extend this expression to the whole circle $\sS^1=I/\{-1\sim
1\}$, where $I=[-1,1]$, in the following way.  Let $\tilde{f}$ be an
orientation-preserving expanding map of $\sS^1$ such that
$\tilde{f}(0)=0$ and $\tilde{f}'>\tilde\sigma$ for some constant
$\tilde\sigma>>1$. We define
$\epsilon=2\cdot x_{k_0}/(1+e^{-\pi/\beta})$,
so that $x_{k_0}$ is the middle point of the interval
$(e^{-\pi/\beta}\epsilon,\epsilon)$ and
fix two points $x_{k_0}<\hat{y}<\tilde{y}<\epsilon$, with
\begin{equation}\label{eq:condhaty}
|\hat{f}'(\hat{y})|>>1\quad\mbox{and also}\quad
2 \frac{1-\epsilon^\tau}{1+e^{-\pi/\beta}} x_{k_0}
> \hat y >  x_{k_0},
\end{equation}
where $\tau$ is a small positive constant and we take
$k_0=k_0(\tau)$ sufficiently big (and $\epsilon$ small
enough) in order that \eqref{eq:condhaty} holds.  Then
we take $f$ to be any smooth map on $S^1$ coinciding
with $\hat{f}$ on $[-\hat{y},\hat{y}]$, with
$\tilde{f}$ on $S^1\setminus[-\tilde{y},\tilde{y}]$,
and monotone on each interval $\pm[\hat{y},\tilde{y}]$.

Finally let $f_t$ be the following one-parameter family of circle
maps unfolding the dynamics of $f=f_0$

\begin{equation}
\label{e2.4,8}
f_{t}(z)=\left\{
\begin{array}{ll}
f(z)+t & \mbox{   for   } z\in (0,\epsilon]\\
f(z)-t & \mbox{   for   } z\in [-\epsilon,0)
\end{array}
\right.
\end{equation}
for $t\in(-\epsilon,\epsilon)$. For
$z\in\sS^1\setminus[-\epsilon,\epsilon]$ we assume
only that $\big|\frac{\partial}{\partial z}
f_t(z)\big|\ge2$.

From the works \cite{PRV98} together with \cite{ArPa04},
it is known that for a positive Lebesgue measure subset $P$
of parameters $t$ the map $f_t$ has a dense orbit, is
non-uniformly expanding with slow recurrence to the
critical set $\SC=\{0\}\cup\{x_k: |k|\ge k_0\}$,
admits a unique absolutely continuous invariant
probability measure $\mu_t$ and the corresponding tail set
$\Gamma_n^t$ satisfies $\lambda(\Gamma_n^t)\le C
e^{-\xi n}$ for some constants $C,\xi>0$; see
\cite[Theorem A]{PRV98} and \cite[Theorems A, B and C]{ArPa04}.

Hence, from Remark~\ref{rmk:Gamma_tail} we can apply
Theorem~\ref{principal} to each of these maps $f_t$.

\begin{corollary}
  \label{cor:infinite-modal}
  Given $t_0\in P$, the map $f=f_{t_0}$ is
  stochastically stable for the adapted family of
  random perturbations $(f_t,\theta_\epsilon^\N)$
  obtained according to Theorem~\ref{principal}.
\end{corollary}

This is the first result on stochastic stability of
one-dimensional maps with infinitely many critical points.

\subsection{Non-uniformly expanding quadratic maps}
\label{sec:non-uniformly-expand-1}

The quadratic family $f_{a}:[-1,1]\rightarrow[-1,1]$
given by $f_{a}=1-ax^{2}$ for $0<a\leq 2$ provides a
class of maps satisfying the hypothesis of
Theorem~\ref{principal}. Indeed, Jakobson~\cite{Ja81} and
Benedicks-Carleson~\cite{BC85} prove the existence of a
physical measure for a positive Lebesgue measure subset
of parameters $a\in (0,2]$ for which $f_a$ is
non-uniform expanding with slow recurrence to the
critical point; Young~\cite{Yo92} and,
more recently, Freitas~\cite{freitas} obtain exponential
decay of the tail sets $\Gamma_n$. From
Remark~\ref{rmk:Gamma_tail} we can apply
Theorem~\ref{principal} for all the maps in the
positive Lebesgue measure subset of parameters found by
Jacobson and Benedicks-Carleson, obtain stochastic stability
for this class of maps. We note that strong stochastic
stability was obtained for the same class in the work of
Baladi-Viana~\cite{BaV96}.

%%%%%%%%%%%%%%%%%%%%%%%%%%%%%%%%%%%%%%%%%%%%%%%%%%%%

\subsection{Intermitent Maps}
\label{sec:intermitent-maps}

Our results enables us also to deduce stochastic
stability for a class of intermittent applications
\cite{manneville1980}, where this property was obtained
for maps $C^{1+\alpha}$ but with the condition that
$\alpha\geq 1$; see \cite{ArTah}. Recently Shen, together with van
Strien in \cite{ShvS13}, obtained strong stochastic stability
for the Manneville-Pomeaux family of intermittent maps,
answering the questions raised in \cite{ArTah}.

Consider $\alpha>0$ and the map $T_\alpha: [0,1]\rightarrow [0,1]$ given by:
$$
T_\alpha(x)=\left\{
\begin{array}{ccc}
x+2^{\alpha}x^{1+\alpha}, & \mbox{~if~} & x\in[0,\frac{1}{2})\\
x-2^{\alpha}(1-x)^{1+\alpha}, & \mbox{~if~} & x\in[\frac{1}{2},1].
\end{array}
\right.
$$
This map is a $C^{1+\alpha}$ local diffeomorphism of
$\sS^1:=[0,1]/\{0\sim 1\}$, so there are no critical
points. The unique fixed point is $0$ with
$DT_\alpha(0)=1$. If $\alpha\geq 1$, then the Dirac mass in
zero $\delta_{0}$ is the unique physic probability measure
and so the Lyapunov exponent in Lebesgue almost every point
is zero; see \cite{thaler1983}. But, for $0<\alpha<1$, there
exists a unique absolutely continuous invariant probability
$\mu$ which is physical and whose basin has full Lebesgue
measure.  To deduce stochastic stability for $\alpha$ in a
subinterval of $(0,1)$, we need some definitions and
results.

Given a $T_\alpha$-invariant and ergodic probability measure
$\mu$ and $\epsilon>0$ we define the \emph{large deviation
  in time $n$ of the time average of the observable
  $\varphi$ from its spatial average} as
  \begin{align*}
    \mathrm{LD}_{\mu}(\varphi,\epsilon,n)
    =
    \mu\left\{x : \left|\frac{1}{n}\sum_{i=0}^{n-1}\varphi(
        f^{i}(x))-\int\varphi d\mu\,\right|>\epsilon\right\}
  \end{align*}
We note that  Birkhoff's Ergodic Theorem ensures
$\mathrm{LD}_{\mu}(\varphi,\epsilon,n)\xrightarrow[n\rightarrow\infty]{}0$
and the rate of this convergence is a relevant quantity.

Since $T_\alpha$ is a local diffeomorphism we have
$\Gamma_{n}=\{x\in\sS^1: \SE(x)> n \}$ and this is naturally
a deviation set for the time averages of
$\varphi=\log|DT_\alpha|$: if $\mu_\alpha$ is the unique
absolutely continuous $T_\alpha$-invariant probability, then
the Lyapunov exponent $\lambda=\int\varphi\,d\mu>c$, where
$c>0$ is the constant in the definition of non-uniform
expansion (\ref{nue}), and so for all large enough $n>1$ and
small enough $\epsilon>0$
\begin{align}\label{eq:LD-Gamma}
  \mathrm{LD}_{\mu}(\log|DT_\alpha(x)|,\epsilon,n)\ge\mu(\Gamma_{n}).
\end{align}
To estimate $\mu(\Gamma_n)$ we now relate
$\mathrm{LD}_{\mu}$ with the rate of decay of correlations.
Let \( \mathcal B_{1}, \mathcal B_{2} \) denote Banach
spaces of real valued measurable functions defined on
\( M \).  We denote the \emph{correlation} of non-zero
functions $\varphi\in \mathcal B_{1}$ and \( \psi\in
\mathcal B_{2} \) with respect to a measure $\mu$ as
\begin{align*}
  \mathrm{Cor}_\mu  (\varphi,\psi)
  =
  \frac{1}{\|\varphi\|_{\mathcal
      B_{1}}\|\psi\|_{\mathcal B_{2}}}\left|\int
    \varphi\, \psi\, d\mu-\int \varphi\, d\mu\int
    \psi\, d\mu\right|.
\end{align*}
We say that we have \emph{decay of correlations}, with
respect to the measure $\mu$, for observables in
$\mathcal B_1$ \emph{against} observables in $\mathcal
B_2$ if, for every $\varphi\in\mathcal B_1$ and every
$\psi\in\mathcal B_2$ we have
\begin{align*}
  \mathrm{Cor}_\mu(\varphi,\psi\circ f^n)\xrightarrow[n\rightarrow\infty]{}0.
\end{align*}

The following result from~\cite{Melb09} allows us to relate
decay of correlations with large deviations; see also
\cite{AFLV11}. We say that a measure $\mu$ is
\emph{$f$-non-singular} if for all measurable sets $A$ such
that $\mu(A)=0$, then $\mu(f^{-1}(A))=0$.

\begin{theorem}[\cite{Melb09,AFLV11}]\label{exemplo 2}
  Let \( f: M\to M \) preserve an ergodic probability
  measure \( \mu \) with respect to which $f$ is
  non-singular. Let \( \SB\subset L^{\infty}(\mu) \) be a
  Banach space with norm \( \|\cdot \|_{\SB} \) and
  $\varphi\in \SB$.
 Let $\beta>0$ and suppose that there exists $\kappa>0$
 such that for all $\psi\in
 L^\infty(\mu)$ we have
 \(
 \mathrm{Cor}_\mu(\varphi,\psi\circ f^n) \le \kappa\cdot n^{-\beta}.
 \)
 Then, for every
 \( \epsilon>0 \), there exists  \( C=C(\varphi, \epsilon)>0 \)
 such that \(
 \mathrm{LD}_{\mu}(\varphi,\epsilon,n) \leq  C n^{-\beta}.
 \)
\end{theorem}

We now observe that the absolutely continuous
$T_\alpha$-invariant probability measure $\mu_\alpha$
$f$-non-singular and that the following estimate for the
rate of decay of correlations is known.

\begin{theorem}[Theorem 4.1 in \cite{LSV99}]
  For all $\psi\in L^{\infty}$ and $\varphi\in C^{1}([0,1])$
  such that $\int \varphi d\mu=0$ we have: $\left|\int
    (\psi\circ T_\alpha^{n})\cdot\varphi\, d\mu\right| \le
  A(\|\varphi\|_{C^1})\cdot \|\psi\|_{\infty}\cdot
  n^{1-1/\alpha}(\log n)^{1/\alpha}$, where
  $A:\mathbb{R}\rightarrow\mathbb{R}$ is an affine map.
\end{theorem}

Hence, since $\log|DT_\alpha(x)|$ is a bounded continuous
function on $[0,1]$, there is a constant $C>0$ such that
$$
\mathrm{Cor}_{\mu}(\log|DT_\alpha(x)|,\psi\circ
T_\alpha^{n}) < Cn^{1-1/\alpha}.\log n^{1/\alpha}.$$ From
Theorem~\ref{exemplo 2} and relation~(\ref{eq:LD-Gamma}) we
deduce that, for every $\delta>0$, we have a constant
$C_1>0$ such that
\begin{eqnarray}\label{eq:mu-gamma}
\mu(\Gamma_{n})<C_{1}\cdot n^{(1-1/(\alpha+\delta))}.
\end{eqnarray}
Since $\mu\ll\lambda$, we have $d\mu=h\,d\lambda$ with a
density function $h$ which, from \cite[Theorem A]{Hu04}, is
bounded, strictly positive and, for a neighborhood $I_0$ of
$0$ there are constants $R>0$ and $\sigma_{0}=
\lim_{x\rightarrow 0}\sum_{x_{1}\in
  T_\alpha^{-1}(x)\setminus
  I_{0}}\frac{h(x_{1})}{DT_\alpha(x_{1})}$ such that
$|x^{\alpha}\cdot h(x)-\sigma_0|\leq R\cdot x^{\alpha}$.
This enables us to find $\kappa>0$ such that
$\lambda(\Gamma_n)\le \kappa \mu(\Gamma_n)$ which, together
with~(\ref{eq:mu-gamma}) provides a constant $C>0$ such that
for all small $\delta>0$ and large $n$
$$\lambda(\Gamma_{n})<C \cdot n^{(1-1/(\alpha+\delta))}.$$
We therefore have for $p>3$, since $\delta>0$ may be take arbitrarily small
\begin{align*}
  \sum_{n=1}^{\infty}n^{p}\cdot\lambda(\Gamma_{n})
  <
  C\sum_{n=1}^{\infty}n^{(p+1-1/(\alpha+\delta))}<\infty
  \quad\text{for all}\quad 0<\alpha\le\frac{1}{p+2}.
\end{align*}
Thus, for any $p>3$, we get for $0<\alpha<\frac{1}{5}$ the
$L^p$ integrability of the first hyperbolic time  map with respect
to $\lambda$ and, from Theorem~\ref{principal} we obtain

\begin{corollary}\label{cor:intermittentstochstab}
  All intermitent maps $T_\alpha$ with parameters
  $0<\alpha<\frac{1}{5}$ are stochastically stable under
  adapted random perturbations.
\end{corollary}

\section{Adapted random perturbations}
\label{sec:adapted-random-pertu}

Here we construct adapted random perturbations. These
perturbations are constructed by an adequate choice of
hyperbolic times along almost all orbits. Then we show that
these specially chosen hyperbolic times are preserved under
the adapted random perturbations in such a way that the
random map is non-uniformly expanding and has
slow-recurrence for random orbits. In addition, the
hyperbolic times for a point $(\underline t,x)\in\Omega$
under the adapted random perturbations are the same as the
hyperbolic times of $x$ for the unperturbed dynamics.

The only assumption is that the original unperturbed map
admits a pair $(\sigma,\delta)$, with $0<\delta,\sigma<1$,
satisfying: the first $(\sigma,\delta)$-hyperbolic time map
$h$ is defined $\lambda$-almost everywhere and $h$ is
$L^p$-integrable for some $p>3$, i.e.,
$\sum_{n\ge1}n^p\lambda(h^{-1}(n))<\infty$.

In what follows we fix $(\sigma,\delta)$ as above and write
hyperbolic time to mean $(\sigma,\delta)$-hyperbolic time.

\begin{definition}\label{th adapt} The adapted hyperbolic time of
  $x\in M\setminus\mathcal{C}$ is the number
$$H(x):=\sup\{h(z)-l;
x=f^l(z), z\in M ~\mbox{and}~l\geq 0\}$$ where
$h:M\rightarrow\mathbb{Z}^{+}$ is the first hyperbolic time
function.
\end{definition}

Note that $H(x)$ is a hyperbolic time for $x$. In fact, if
$x=f^l(z)$ for $l\ge1$ and some point $z$, and $h(z)$ is the
first hyperbolic time of $z\in M$, then $h(z)-l$ is a
hyperbolic time for $f^{l}(z)=x$. Moreover, it is clear that
$H(x)\ge h(x)$ if $h(x)$ is finite.

To check that $H$ is finite almost everywhere, we note that
\begin{eqnarray}\label{controle do tempo adaptado}
H(x)\leq\sup\left\{n\in\mathbb{Z}^{+}:
x\in\bigcup_{i=0}^{n-1} f^{i}(h^{-1}(n))\right\}.
\end{eqnarray}

  Since for a one-dimensional map $f$ we have $|\det
  Df|=\|Df\|=|Df|$, then the assumption $h\in L^{p}(\lambda)$
  with $p>3$ implies (\ref{eq:summa}) in the one-dimensional
  setting.

\begin{lemma}
  \label{le:summa1d}
  Let $f$ be a non-uniformly expanding one-dimensional map
  having slow recurrence to the non-degenerate critical set.
  Let us assume that the first hyperbolic time map satisfies
  $h\in L^p(\lambda)$ for some $p>3$. Then $ \sum_{n\geq
    1}\sum_{j=0}^{n-1}\lambda(f^{j}(h^{-1}(n)))<\infty.$
\end{lemma}

\begin{proof}
  We follow~\cite[Section 3]{AlCasPin}. We note that if $n\ge1$
  is a $(\sigma,\delta)$-hyperbolic time, then $|\det
  Df^n(x)|\ge a_n=\sigma^{-n}$. Let $q(x)=\min\{k\ge1:|\det
  Df^k(x)|\ge a_k\}$. Then $q(x)\le h(x)$ and so $q\in
  L^p(\lambda)$ if $h\in L^p(\lambda)$.

  Let $W_n=\{x\in M:q(x)>n\}$. Then
  $W_n\subset\cup_{m>n}h^{-1}(m)$ and so we can find
  constants $\kappa,C>0$ such that
  \begin{align*}
    \lambda(W_n) \le \sum_{m>n}\lambda(h^{-1}(m)) \le
    \sum_{m>n}\frac{\kappa}{m^{p}} \le \frac{C}{n^{p-1}}.
  \end{align*}
  Hence there exists $\beta>0$ and $N\in\N$ such that
  $b_n=n^\beta$ satisfies
  $b_n\le\min\{a_n,\lambda(W_n)^{-\epsilon}\}$ for all $n\ge
  N$ and some $0<\epsilon<\frac{p-3}{p-1}$. In addition, we
  clearly have $b_nb_k\ge b_{k+n}$ for all big enough
  $k,n\in\N$.
  In this setting, $U_n=\{x\in M:|\det Df^n(x)|\ge b_n\}$
  is such that
  \begin{itemize}
  \item $\cup_{n\ge1} U_n$ has full Lebesgue measure, since
    $T_n=\{x\in M: n$ is a $(\sigma,\delta)$-hyperbolic
    time$\}$ satisfies $h^{-1}(n)\subset T_n\subset U_n$; and
  \item  if
    $x\in U_n$ and $f^n(x)\in U_m$, then $x\in U_{n+m}$
  \end{itemize}
(i.e., $(U_n)_{n\ge1}$ is a \emph{concatenated collection} as
  defined in \cite{AlCasPin}). In addition, letting
  $\hat q(x)=\min\{n\ge1:x\in U_n\}$, we have again $\hat
    q(x)\le h(x)$ in general. However, if $f$ is one-dimensional,
    then we obtain equality $\hat q(x)=h(x)$.

    The choices of $U_n$ and the sequence $b_n$ ensure that
    $\sum_{n\ge n_0}\sum_{j=0}^{n-1}\lambda\big(f^j(\hat
    q^{-1}(n))\big)<\infty$; see \cite[Section 3]{AlCasPin}.
    Moreover, in the one-dimensional setting, this series
    coincides with the one in the statement of the lemma.
\end{proof}

Under this summability condition we obtain the following.

\begin{lemma}[Lemma 2.1 in \cite{AlCasPin}]
\label{tempo adaptado finito}
If (\ref{eq:summa}) is true, then $H(x)<\infty$ to
$\lambda$-almost every $x\in M$.
\end{lemma}
\begin{proof}
  For $\lambda$-almost every $x\in M$ we consider the set
  $\mathcal{K}(x)=\{f^{j}(x)\}_{j=0}^{h(x)-1}$, which we
  call a \emph{chain}. Suppose that for some $z\in M$ we
  have that $z$ belongs to infinitely many chains
  $\mathcal{K}_{j}(x_{j})=\{x_{j},f(x_{j}),\ldots,f^{s_{j}-1}(x_{j})\}$
  for $j\geq 1$ where $s_{j}=h(x_{j})$ is the first
  hyperbolic time for $x_{j}$ and $s_{j}\rightarrow \infty$.
  
  Now for each $j\geq 1$ we take $1\leq r_{j}<s_{j}$
  such that $z=f^{r_{j}}(x_{j})$ and claim that $\lim
  r_{j}=\infty$. Indeed, otherwise, taking a subsequence of
  $r_{j}$, we can assume that there is $N>0$ such that
  $r_{j_{k}}<N$, $\forall k\geq 1$. But this implies that
  $x_{j}\in\cup_{i=1}^{N}f^{-i}(z)$, $\forall j\geq 1$ and
  so the number of elements of $\cup_{i=1}^{N}f^{-i}(z)$ is finite:
  $\#(\cup_{i=1}^{N}f^{-i}(z))<\infty$. However we are
  assuming that the number of chains is infinite. This
  contradiction proves the claim.

Hence $r_{j}\rightarrow\infty$ and $z=f^{r_{j}}(x_{j})\subset
f^{r_{j}}(h^{-1}(s_{j}))$ and so we get
$$z\in\displaystyle\cup_{n\geq k} \cup_{j=0}^{n-1}f^{j}(h^{-1}(s_{j})),~~
\forall k\geq 0.$$ Since $\sum_{n\geq
  1}\sum_{j=0}^{n-1}\lambda(f^{j}(h^{-1}(n)))<\infty$, we
obtain $\lambda(\cup_{n\geq
  k}\cup_{j=0}^{n-1}f^{j}(h^{-1}(n)))\xrightarrow[k\rightarrow\infty]{}
0$. Then the set of points belonging to infinitely many
chains has null Lebesgue measure. Finally, from relation
$(\ref{controle do tempo adaptado})$ the proof of the lemma
is complete.
\end{proof}

Note that it is not possible ensure that $H(f(x))=H(x)-1$ in
general, because $x$ and $f(x)$ can be in orbits of
different points, namely $z\neq w$ whose first hyperbolic
times do not satisfy the relation $h(w)=h(z)-1$. Then the
adapted hyperbolic time for $f(x)$ can be bigger than
$H(x)-1$.  However, note that $H(f(x))$ can not be smaller
than $H(x)-1$ because $x$ already has $H(x)$ as hyperbolic
time. In any case we have the following important
\emph{monotonicity property}
of our choice of adapted hyperbolic time
\begin{align}\label{eq:adapted-hyp-time}
  H(f(x))\geq H(x)-1.
\end{align}
Similarly we obtain $H(f^{j}(x))\geq H(x)-j$ for $0\le j <
H(x)$ as long as $H(x)$ is finite.

\begin{lemma}[Lemm 5.2 in \cite{ABV00}]\label{lema 5.2}
  Given $\sigma<1$ and $\delta>0$, there is $\delta_{1}>0$
  such that if $n$ is a $(\sigma,\delta)$-hyperbolic time for
  $x\in M\setminus\mathcal{C}$ then there exits a neighborhood
  $V_{n}(x)$ of $x$ such that:
\begin{enumerate}
\item[1.] $f^{n}$ maps $V_{n}$ diffeomorphically into the
  ball of radius $\delta_{1}$ centered at $f^{n}(x)$.
\item[2.] For all $1\leq k<n$ and $y,z\in V_{n}(x)$
$$\mbox{dist}(f^{n-k}(y),f^{n-k}(z))\leq\sigma^{k/2}.\mbox{dist}(f^{n}(y),f^{n}(z)).$$
\end{enumerate}
\end{lemma}

By the definition of hyperbolic time, if $n$ is a
$\sigma$-hyperbolic time for a point $x\in M$, then there are
neighborhoods $V_{n-j}\subset B_{\delta_{1}\sigma^{j}}(f^{j}(x))$ of
$f^{j}(x)$ which are sent in time  $j$ diffeomorphically into the ball
$B_{\delta_{1}}(f^{n}(x))$ for all $0\leq j\leq n$.

\begin{lemma}\label{le:Hconst}
  In our setting, for $\lambda$-almost every $x$, there
  exists an open neighborhood $V_H(x)$ of $x$ such that
  $H\mid V_H(x)$ is constant.
\end{lemma}

\begin{proof}
  The subset $Y$ of $M$ of points having some hyperbolic
  time is such that $\lambda(Y)=1$. Hence $f^{-1}(Y)$ also
  has full $\lambda$-measure since $f$ is a local
  diffeomorphism away from a critical/singular set with zero
  $\lambda$-measure. Therefore $\lambda(\cap_{n\ge1}(Y\cap
  f^{-n}(Y)))=1$ and we conclude that every point in the
  pre-orbit $\cup_{n\ge1}f^{-n}(\{x\})$ of Lebesgue almost
  every point $x$ has some hyperbolic time.

  Let $X$ be the subset of $M$ such that $H(x)<\infty$ for
  all $x\in X$. We know that $\lambda(X)=1$.

  Let us now fix $x\in Y\cap X$. Hence we have $h(y)<\infty$
  for every point $y$ in the pre-orbit of $x$ and,
  moreover, if $x=f^{k}(y)$ then $h(y)-k\le H(x)$ by
  definition of $H(x)$.

  It follows that the neighborhood $V_{h(y)}(y)$ of $y$
  associated to the hyperbolic time $h(y)$ is such that
  $f^k(V_{h(y)}(y))\supset V_{H(x)}(x)$, since $h(y)-k\le
  H(x)$.

  Therefore, for $x'\in V_{H(x)}(x)\subset f^k(V_{h(y)}(y))$
  the inverse map $\vfi$ of $f^k\mid V_{h(y)}(y)$ is such
  that $\vfi(x')=y'\in V_{h(y)}(y)$. Thus $h(y')\le h(y)$
  (recall that $h(y')$ is the first hyperbolic time of $y'$
  and $h(y)$ is already a hyperbolic time for $y'$). It
  follows that $h(y')-k\le h(y)-k\le H(x)$.

  This argument is true of any element $y$ of the pre-orbit of
  $x$, whose neighborhood $V_{h(y)}(y)$ is sent by $f^k$ to
  a set covering $V_{H(x)}(x)$. Hence all pre-images of
  points $x'\in V_{H(x)}(x)$ respect the same inequality,
  that is, $H(x')\le H(x)$. But the reverse inequality is
  also true by definition of $H$, since $x'\in V_{H(x)}(x)$
  has $H(x)$ as an hyperbolic time. This completes the proof.
\end{proof}

\begin{remark}
  \label{rmk:thadapt}
  We make the convention that $H(x)=1$ wherever the supremum
  in Definition~\ref{th adapt} is not finite.
\end{remark}

\begin{remark}\label{rmk:HnearC}
  Besides the obvious relation $H(x)\ge h(x)$ almost
  everywhere, we can say more in certain regions. Let us
  assume that $V$ is the largest open neighborhood of the
  critical set $\SC$ such that $|Df\mid (M\setminus
  V)|>\sigma^{-1}$ and $V\cap f^{-1}(V)=\emptyset$. Then
  $H=h$ in $V$, since $h(x)\ge2$ for almost all points $x\in
  V$ and all pre-orbits of $x$ have $1$ as a first
  $\sigma$-hyperbolic time, which is smaller that $h(x)-1$.

  The above conditions on a neighborhood of the critical set
  are easily checked for non-uniformly expanding quadratic
  maps and, by \cite[Section 4]{PRV98}, this is also true
  for the infinite-modal family $f_\mu$ at every parameter
  of the positive Lebesgue measure subset $P$; see
  Section~\ref{sec:examples-application}.
\end{remark}

%%%%%%%%%%%%%%%%%%%%%%%%%%%%%%%%%%%%%%%%%%%%%%%%%%%%%%%%%%%%%%

\subsection{Preservation of hyperbolic times}
\label{sec:preserv-hyperb-times}

Now we show that hyperbolic times are preserved if we define
a random perturbation adapted to the structure of hyperbolic
times using $H$, as in (\ref{pertaditiva}) with
$\zeta(x)=\xi e^{-\eta H(x)^2}$ for suitably chosen
constants $\xi,\eta>0$. We first define the notions of
hyperbolic times and slow recurrence in our random setting.

\subsubsection{Random non-uniformly expanding maps and random
  slow recurrence}
\label{sec:random-non-uniformly}

We now define the analogous notions of  non-uniform
expansion and slow recurrence for random dynamical systems
in our setting.

\begin{definition}\label{nue para o caso aleatorio}
  We say that a map is non-uniformly expanding map for
  random orbits if there exists a constant $c>0$ such that
  for $\epsilon>0$ sufficiently small and
  $\theta^{\mathbb{N}}_{\epsilon}\times
  \lambda$-a.e. $(\underline{t},x)$ we
  have
$\limsup_{n\to+\infty}\frac1n
\sum_{j=0}^{n-1}\log\| Df(f^{j}_{\un t}(x))^{-1}\|\leq -c<0.$
\end{definition}

\begin{definition}\label{rec sist alea}
  We say that a random dynamical system
  $(f_{\underline{t}},\theta_{\epsilon})$ has slow
  recurrence to the critical set for random orbits if, for
  all small enough $\gamma>0$,  there exists
  $\delta>0$ such that $\theta^{\mathbb{N}}_{\epsilon}\times
  \lambda$-a.e. $(\underline{t},x)$ we have
$\limsup_{n\to+\infty}\frac{1}{n}\sum_{j=0}^{n-1}-\log
d_{\delta}(f_{\un t}^{j}(x),\SC)\leq \gamma.$
\end{definition}

\subsubsection{Random hyperbolic times}
\label{sec:random-hyperb-times}

An definition of hyperbolic analogous to \ref{def:hyptimes}
can be made for the random system
$(f_{t},\theta_{\epsilon})$.

\begin{definition}[Random Hyperbolic Time]
\label{def:hyptimesrand}
Given $\sigma\in(0,1)$ and $\delta>0$, we say that n is a
$(\sigma,\delta)$-hyperbolic time for a point
$(\underline{t},x)\in \Omega\times M$ if:
\begin{align*}
  \prod_{j=n-k}^{n-1}\|
  Df_{t_{j+1}}(f_{\underline{t}}^{j}(x))^{-1}\|
  \leq\sigma^{k}
  \quad\mbox{and}\quad
  d_{\delta}(f_{\underline{t}}^{n-k}(x),\SC)\geq\sigma^{bk},
  \quad\text{for all}\quad 1\leq k\leq n.
\end{align*}
\end{definition}

\begin{theorem}\label{nuero}
  If $f$ is non-uniformly expanding with slow recurrence to
  the critical set in the interval or the circle, then for
  each $\delta>0$ there is
  $\zeta:M\rightarrow\mathbb{R}^{+}$ mensurable and locally
  constant such that the adapted random perturbation
  $(\ref{pertaditiva})$ satisfies: there exists
  $0<\sigma<\hat\sigma <1$ such that for
  $\lambda$-almost every point $x$ and all
  $\underline t\in[-1/2,1/2]^\nat$ has $H(x)$ as
  $(\hat\sigma,\delta)$-hyperbolic time.
\end{theorem}

We assume that $f$ has a non-degenerate critical set
$\mathcal C$. We also assume without loss of generality in
what follows that $B\delta^{1-\beta}\le\log\sigma^{-1/2}$
and $\delta_1=\frac12\delta\le\frac12$, where $B,\beta>0$
are given in the non-degeneracy conditions of $\mathcal
C$.

\begin{remark}\label{rmk:localdiffeo}
  The same arguments and constructions presented in
  this section enable us to trivially obtain a version
  of Theorem~\ref{nuero} for the local diffeomorphism
  case, that is, the case where there are no critical
  (or singular) points: $\mathcal C=\emptyset$.
\end{remark}

\begin{remark}
  \label{rmk:hbdd}
  Since by construction $h(x)\le H(x)$, whenever $h(x)$ is
  finite, then we have for $\lambda$-a.e. $x$ that
  $V_H(x)\subset V_n(x)$ for all hyperbolic times $n$ of $x$
  such that $h(x)\le n\le H(x)$. 

  Moreover, we have that the random orbit of $(\un t,x)$ has
  the same hyperbolic times $n$ of the unperturbed orbit of
  $x$ as long as $h(x)\le n\le H(x)$. In particular, the
  first hyperbolic time of $(\un t,x)$ is given by $h(x)$.
\end{remark}

\begin{lemma}
  \label{le:bdd-der-above}
  There exists $\omega>\sigma^{-1/2}$ such that, if $n$ is a
  $(\sigma,\delta)$-hyperbolic time for $x$, then
  $\|Df^n(x)\|\le\omega^n$. 
\end{lemma}

\begin{proof}
  Using the non-degenerate condition (S1) we get
  $\log\|Df(x)\|\le\log B-\beta\log d(x,\SC)$. Hence,
  since $n$ is a hyperbolic time, we have from their
 construction that they satisfy
 (\ref{eq:hip-times-prop}) which implies
 \begin{align*}
   \log\|Df^n(x)\|
   &\le
   \sum_{j=0}^{n-1}\log\|Df(f^j(x))\|
   \le
   n\log B-\beta\sum_{j=0}^{n-1} \log d(f^j(x),\SC)
   \\
   &\le
   \log B^n +\beta\sum_{j=0}^{n-1} -\log
   d_\delta(f^j(x),\SC)
   +\beta\sum_{d(f^j(x),\SC)\ge\delta} -\log d(x,\SC)
   \\
   &\le
   \log B^n +\beta\epsilon n -\beta n \log\delta
   =
   n(\log B +\beta(\epsilon-\log\delta))
 \end{align*}
 and so $\|Df^n(x)\|\le\omega^n$, where $\omega=\max\{\log B
 +\beta(\epsilon-\log\delta), \sigma^{-1/2}\}$. 
\end{proof}

\begin{lemma}\label{le:innerball}
  If $n$ is a $(\sigma,\delta)$-hyperbolic time for $x$, then
  $B_{\delta_{1}\omega^{-(n-j)}}(f^{j}(x))\subset
  V_{n-j}(f^j(x))\subset
  B_{\delta_1\sigma^{(n-j)/2}}(f^{j}(x))$ for each $0\leq
  j\leq n$.
\end{lemma}

This result is essential to show that to keep the hyperbolic
time under perturbation all that we need is to maintain the
random orbits within a certain distance to the unperturbed
orbit during the iterated of the adapted hyperbolic time.

\begin{proof}[Proof of Lemma~\ref{le:innerball}]
  We have
  $d_{\delta}(f^{j}(x),\mathcal{C})\geq\sigma^{b(n-j)}$
  for all $0\leq j\leq n$ and so either
  $d(f^{j}(x),\mathcal{C})\geq\sigma^{b(n-j)}$ with
  $f^j(x)\in B_\delta(\mathcal C)$, or
  $d(f^{j}(x),\mathcal{C})\geq\delta$.

  Hence for $y\in B_{\delta_{1}\sigma^{(n-j)/2}}(f^{j}(x))$
  we have either
  $\frac{d(y,f^{j}(x))}{d(f^{j}(x),\mathcal{C})}\le
  \delta_1\sigma^{(1/2-b)(n-j)}\le\frac12$ or
  $\frac{d(y,f^{j}(x))}{d(f^{j}(x),\mathcal{C})}\le
  \frac{\delta_1}{\delta}\sigma^{(n-j)/2}\le\frac12$ for
  $0\leq j\leq n$ (recall that $0<b\le1/2$ from the
  definition of non-degenerate critical set).  This enables
  us to use non-degeneracy conditions (S1) and (S2).

For $y\in B_{\delta_{1}\sigma^{(n-j)/2}}(f^{j}(x))$
% \begin{eqnarray}\label{S1}
% -B\frac{d(y,f^{j}(x))}{d(f^{j}(x),\mathcal{C})^{\beta}}
% \leq
% \log\frac{\p Df(y)^{-1}\p}{\p Df(f^{j}(x))^{-1}\p}
% \leq
% B\frac{d(y,f^{j}(x))}{d(f^{j}(x),\mathcal{C})^{\beta}}
% \end{eqnarray}
since $b\beta\le1/2$, the value of
$B\frac{d(y,f^{j}(x))}{d(f^{j}(x),\mathcal{C})^{\beta}}$ is
bounded above by either
$B\delta_1\sigma^{(1/2-b\beta)(n-j)/2}$ or
$B\delta_1\delta^{-\beta}\sigma^{(n-j)/2}=\frac{B}2\delta^{1-\beta}\sigma^{(n-j)/2}$,
and both are smaller than $\log\sigma^{-1/2}$.  Thus from
(S2) for all $y\in B_{\delta_{1}\sigma^{(n-j)/2}}(f^{j}(x))$
\begin{align}\label{eq:df-y}
   \sigma^{1/2}\|Df(f^{j}(x))^{-1}\| \le\| Df(y)^{-1}\|\le\sigma^{-1/2}\|
Df(f^{j}(x))^{-1}\|.
\end{align}
For $j= n-1$ above, we get for every $y\in
B_{\delta_{1}\sigma^{1/2}}(f^{n-1}(x))$
$$\sigma^{1/2}=\sigma^{-1/2}\|Df(f^{n-1}(x))^{-1}\|
\ge\|Df(y)^{-1}\|\ge\sigma^{1/2}\|Df(f^{n-1}(x))^{-1}\|\ge\sigma^{3/2}.$$
Hence, a smooth curve $\gamma$ from $f^n(x)$ to the boundary
of $B_{\delta_1}(f^n(x))$ and inside this ball must be such
that the unique curve $\tilde\gamma$ contained in
$V_{1}(f^{n-1}(x))$ such that $f^{n-1}(x)\in\tilde\gamma$
and $f(\tilde\gamma)=\gamma$ satisfies
$\sigma^{3/2}\delta_1=\sigma^{3/2}\ell(\gamma)\le\ell(\tilde\gamma)\le
\sigma^{1/2}\ell(\gamma)=\delta_1\sigma^{1/2}$, where
$\ell(\cdot)$ denotes the length of any smooth curve on $M$
and, recall, $f^{n-j}\mid
V_{n-j}(f^{j}(x)):V_{n-j}(f^{j}(x))\to B_{\delta_1}(f^n(x))$
is a diffeomorphism for all $j=0,\dots,n-1$. Thus
$B_{\delta_1\sigma^{1/2}}(f^{n-1}(x))\supset
V_{1}(f^{n-1}(x))\supset
B_{\delta_1\sigma^{3/2}}(f^{n-1}(x))$. In particular this
shows that the statement of the Lemma is true for
$n=1$, since $\omega>\sigma^{-1/2}$.

Now we argue by induction assuming the Lemma to be true for
all hyperbolic times up to some $n\ge1$ and consider $x$
having $n+1$ as a hyperbolic time. Then for each $1\leq j <n$
\begin{align*}
  B_{\delta_{1}\omega^{-{n-j}}}(f^{j}(x))\subset
  V_{n-j}(f^j(x))\subset
  B_{\delta_1\sigma^{(n-j)/2}}(f^{j}(x))
\end{align*}
since $f(x)$ has $n$ as a hyperbolic time.  For all $y\in
V_{n+1}(x)\cap B_{\delta_1\sigma^{(n+1)/2}}(x)$ we have $f(y)\in
V_1(f(x))$ and so by the induction assumption together with (\ref{eq:df-y})
\begin{align*}
  \|Df^{n+1}(y)^{-1}\|
  \le\prod_{i=0}^{n}\|Df(f^i(y))^{-1}\|
  \le\prod_{i=0}^{n}(\sigma^{-1/2}\|Df(f^i(x))^{-1}\|)
  \le \sigma^{(n+1)/2}
\end{align*}
Therefore, for any smooth curve $\gamma$ from $f^{n+1}(x)$
to the boundary of $B_{\delta_1}(f^{n+1}(x))$ and inside
this ball we have that the unique curve $\tilde\gamma$ contained in
$V_{n}(x)\cap B_{\delta_1\sigma^{(n+1)/2}}(x)$ such that $x\in\tilde\gamma$
and $f^{n+1}(\tilde\gamma)=\gamma$ satisfies
$\ell(\tilde\gamma)\le
\sigma^{(n+1)/2}\ell(\gamma)=\delta_1\sigma^{(n+1)/2}$. Hence
$V_{n+1}(x)\subset B_{\delta_1\sigma^{(n+1)/2}}(x)$. 

Finally, from Lemma~\ref{le:bdd-der-above}, we obtain
for the same curves $\gamma,\tilde\gamma$ as above
$\ell(\gamma)=\ell(f^{n+1}\circ\tilde\gamma)\le\omega^{n+1}\ell(\tilde\gamma)$,
or
$\ell(\tilde\gamma)\ge\omega^{-(n+1)}\ell(\gamma)$. Since
this holds for any smooth curve $\gamma$ from
$f^{n+1}(x)$ to the boundary of
$B_{\delta_1}(f^{n+1}(x))$ and inside this ball, we
conclude that $V_{n+1}(x)$ contains
$B(x,\delta_1\omega^{-(n+1)})$. This completes the
inductive step and concludes the proof.
\end{proof}

\begin{remark}
  \label{rmk:bdderivative}
  From condition (S1) we obtain using the estimate (\ref{eq:df-y})
  \begin{align*}
    |Df(y)^{-1}|\ge\sigma^{1/2}|Df(f^j(x))^{-1}|\ge\frac{\sigma^{1/2}}B
    d(f^j(x),\mathcal C)^{\beta}
    \ge \frac{\sigma^{1/2}}B \sigma^{\beta b (n-j)}
    \ge \frac{\sigma^{1/2}}B \sigma^{(n-j)/2}
  \end{align*}
because $b\beta\le1/2$. Then we arrive at
\begin{align*}
  |Df(y)|\le C\sigma^{-(n-j)/2}, \quad y\in V_{n-j}(f^j(x))
\end{align*}
where $C=B\sigma^{-1/2}$, whenever $x$ has $n\ge1$ as an
hyperbolic time and $0\le j < n$.
\end{remark}

\begin{proposition}\label{pr:rand-hyp-inside}
  Let $f$ is a $C^2$ non-uniformly expanding
  endomorphism having slow recurrence to the critical
  set. There exist constants $\xi,\eta>0$ such that for
  $\zeta(x)=\xi \omega^{-\eta H(x)^2}$ and the family
  $f_t(x)=f(x)+t\cdot\zeta(x)$, if $x$ is such that
  $H(x)$ is a hyperbolic time, then we have
  $f_{\underline{t}}^{j}(x)\in V_{H(x)-j}(x)$ for all
  $0\leq j\leq H(x)$ and each
  $\underline{t}\in\Omega\subset[-1/2,1/2]^{\mathbb
    N}$.

  In particular, $H(x)$ is a
  $(\hat\sigma,\delta)$-hyperbolic time for
  $(\underline t,x)\in\Omega\times M$ whenever
  $H(x)<\infty$, for a constant
  $0<\sigma<\hat\sigma<1$.

  Moreover, if
  $\Omega\subset[-\epsilon_0,\epsilon_0]^\nat$ for some
  $0<\epsilon_0<1/2$ and $H(x)<\infty$, then
  $f_{\underline{t}}^{j}(x)\in
  B_{\epsilon_0\delta_1\omega^{-\eta(H(x)-j)}}(f^j(x))$
  for each $0\leq j\leq H(x)$ and for all $\underline
  t\in\Omega$.
\end{proposition}

\begin{proof}
  Let $\eta>3/2$ be big enough such that
  $\max\{C\sigma^{2\eta-1/2},\sigma^{2\eta}\}<1/2$, choose
  $\xi=\min\{\delta_1/2,1/2\}$ and fix $\underline
  t=(t_1,t_2,\dots)\in\Omega$. Then
  \begin{align*}
    |f_{t_{1}}(x)-f(x)|\leq|t_{1}\zeta(x)|<\xi
    \omega^{-\eta H(x)^2}<\delta_1\omega^{-(H(x)-1)} 
  \end{align*}
  and so $f_{t_1}(x)\in
  B_{\delta_1\omega^{-(H(x)-1)}}(f^{H(x)-1}(x))\subset
  V_{H(x)-1}(f(x))$. Observe that there is $z\in
  V_{H(x)}(x)$ such that $f_{t_1}(x)=f(z)$ and so
  $H(f_{t_1}(x))=H(f(z))\ge H(z)-1=H(x)-1$.

  Now we argue by induction on $k$ and assume that for $1\le
  j\le k<n-1$ we have 
  \begin{enumerate}
  \item $f^j_{\underline t}(x)\in
    B_{\xi\omega^{-\eta(H(x)-j)^2}}(f^j(x))\subset
    V_{H(x)-j}(f^j(x))$, and
  \item $H(f^j_{\underline t}(x))\ge H(x)-j$.
  \end{enumerate}
  It is easy to see that this is true for $k=1$.  For
  $j=k+1$ we get, for some $w\in
  B_{\delta_1\omega^{-\eta(H(x)-k)}}(f^k(x))$ in a segment
  between $f^k_{\underline t}(x)$ and $f^k(x)$, according to
  Remark~\ref{rmk:bdderivative}
  \begin{align*}
    |f^{k+1}_{\underline t}(x)-f^{k+1}(x)|
    &\le
    |f_{t_{k+1}}(f^k_{\underline t}(x))-f(f^k_{\underline
      t}(x))|+|f(f^k_{\underline t}(x)) -f(f^k(x))|
    \\
    &\le
    |t_{k+1}\zeta(f^k_{\underline t}(x))|+
    |Df(w)|\cdot|f^k_{\underline t}(x)-f^k(x)|
    \\
    &<
    \xi\omega^{-\eta H(f^k_{\underline t}(x))^2}
    + C\sigma^{-(H(x)-k)/2}\cdot\xi\omega^{-\eta(H(x)-k)^2}
    \\
    &\le
    \xi\omega^{-\eta(H(x)-k)^2}(1+C\sigma^{-(H(x)-k)/2})
    \\
    &=
    \xi\omega^{-\eta(H(x)-k-1)^2}\omega^{-\eta(2(H(x)-k)+1)}
    (1+C\sigma^{-(H(x)-k)/2})
    \\
    &\le
    \xi\omega^{-\eta (H(x)-k-1)^2}
    (\omega^{-\eta(2(H(x)-k)+1)}+C\sigma^{(2\eta-1/2)(H(x)-k)})
    \\
    &\le\xi\omega^{-\eta (H(x)-k-1)^2}.
  \end{align*}
  The last inequality comes from the choice of $\eta$
  and because $H(x)-k\ge1$ and
  $\omega>\sigma^{-1/2}>\sigma^{-1}$. This proves that
  part (1) of the inductive step. Then there exists
  $z\in V_{H(x)}(x)$ such that
  $f^{k+1}(z)=f^{k+1}_{\underline t}(x)$ and so
  $H(f^{k+1}_{\underline
    t}(x))=H(f^{k+1}(z))=H(z)-(k+1)=H(x)-(k+1)$,
  completing the proof of the inductive step.

Now we check that $H(x)$ is still a hyperbolic time for
$(\underline t,x)$. This follows easily from the statement
of Proposition~\ref{pr:rand-hyp-inside} together with the
estimate~(\ref{eq:df-y}) and
Remark~\ref{rmk:Dft_Df}. However we have to relax the
constants: for $1\le k <H(x)$
\begin{align}\label{eq:pert-prod}
  \prod_{j=n-k}^{H(x)-1} |Df_{t_{j+1}}(f^j_{\underline
    t}(x))^{-1}|
  =
  \prod_{j=n-k}^{H(x)-1}|Df(f^j_{\underline t}(x))^{-1}|
  \le
  \prod_{j=n-k}^{H(x)-1}(\sigma^{-1/2}|Df(f^j(x))^{-1}|)
  \le
  \sigma^{k/2}
\end{align}
and
\begin{align}\label{eq:dist-crit-ratio}
  d(f^{H(x)-j}_{\underline t}(x),\mathcal C)
  &\ge
  d(f^{H(x)-j}(x),\mathcal C)-d(f^{H(x)-j}_{\underline
    t}(x),f^{H(x)-j}(x))\ge\sigma^{bj}-\delta_1\sigma^{j/2}
  \\
  &=\sigma^{bj}(1-\delta_1\sigma^{(b-1/2)j})
  \ge
  (1-\delta_1)\sigma^{bj}\nonumber
\end{align}
whenever $d(f^{H(x)-j}_{\underline t}(x),\mathcal C)<\delta$.
Hence $H(x)$ is a $(\hat\sigma,\delta)$-hyperbolic
time, for some $\sigma<\hat\sigma<1$ for all $x$ such
that $H(x)$ is finite.

Up until now, the proof of was done with a fixed
maximum size $1/2$ for the perturbation. If we consider
$\Omega\subset[-\epsilon_0,\epsilon_0]^\nat$ with
$0<\epsilon_0<1/2$, then the size of $t\cdot\zeta(x)$ is
reduced proportionally in all the previous estimates,
so that we obtain the last part of the statement.
\end{proof}

This concludes the proof of Theorem~\ref{nuero}.

\subsection{Asymptotic rates of expansion and recurrence on
  random orbits}
\label{sec:asympt-rates-expans}

As a consequence of preservation of hyperbolic times,
we have the following uniform estimates for the
asymptotic rate of non-uniform expansion and slow
recurrence for random orbits, i.e., the estimates we
obtain do not depend on the perturbation as long as the
perturbation is small enough.

\begin{proposition} \label{pr:RSlow} If $f$ is a
  non-uniformly expanding map with slow recurrence to
  the critical set having a first hyperbolic time map
  $L^p$-integrable for some $p>3$ then there is
  $\epsilon_{0}\in(0,1/2)$ such that, for all
  $0<r<\epsilon_{0}$, for $\lambda$-almost every point
  $x$ and for all $\underline{t}\in[-r,r]^\nat$, we
  have the bound $\liminf_{n\to+\infty} \frac{1}{n}
  \sum_{j=0}^{n-1} -\log
  d_{\delta}(f_{\underline{t}}^{j}(x),\mathcal{C})<2\epsilon$
  and also $\liminf_{n\to+\infty}\frac1n
  \sum_{j=0}^{n-1} \log \|Df(f^j_{\underline
    t}(x))^{-1}\| \le\frac12\log\sigma$.
\end{proposition}

\begin{proof}
  The last limit inferior is clear: since we have
  infinitely many hyperbolic times $H(x)$ for
  $\lambda$-almost every $x$, we also have infinitely
  many hyperbolic times $H(x)$ for $\lambda$-almost
  every $x$ and every $\underline
  t\in[r,r]^\nat$. Hence
  from~(\ref{eq:pert-prod}) we obtain infinitely many
  hyperbolic times $n_1=H(x),
  n_2=n_1+H(f^{n_1}_{\underline t}(x)),
  n_3=n_2+H(f^{n_2}_{\underline t}(x)),\dots$ along the
  random orbit of $x$ with the average rate
  $\frac12\log\sigma$, which implies the stated bound
  for the limit inferior.

  For the limit inferior of slow approximation, we use
  (\ref{eq:dist-crit-ratio}) to write for all $0\le j<H(x)$
\begin{align}\label{distcritica}
\frac{d(f_{\underline{t}}^{j}(x),\SC)}{d(f^{j}(x),\SC)}
% &\geq
% \frac{d(f^{j}(x),\SC)-d(f_{\underline{t}}^{j}(x),f^{j}(x))}
% {d(f^{j}(x), \SC)}\nonumber
% \\
&\geq 
1-\frac{d(f_{\underline{t}}^{j}(x),f^{j}(x))}{d(f^{j}(x),
  \SC)}
\ge 1-r\sigma^{(1/2-b)(H(x)-j)}.
\end{align}
From the definition of $d_\delta$ we can write, since
$0<r<1/2$ and $H(x)$ is a hyperbolic time
\begin{align*}
  \sum_{j=0}^{H(x)-1}
  -\log d_\delta(f_{\underline{t}}^{j}(x),\SC)
  &\le
  \sum_{j=0}^{H(x)-1}
  -\log (1-r\sigma^{(1/2-b)(H(x)-j)})
  +
  \sum_{j=0}^{H(x)-1}
  -\log d_\delta(f^j(x),\SC)
  \\
  &\le
  \sum_{j=0}^{H(x)-1} 2r\sigma^{(1/2-b)(H(x)-j)}
  +
  \epsilon n
  =
  \frac{2r\sigma^{1/2-b}}{1-\sigma^{1/2-b}}+\epsilon n
  \le 2\epsilon n
\end{align*}
if we take $0<r<\epsilon_0<1/2$ small enough.

The bound on the limit inferior follows again from the
existence of infinitely many hyperbolic times along the
orbit of $(\underline t,x)$ for $\lambda$-almost every
$x$ and all $\underline t\in[-r,r]^\nat$.
\end{proof}

%%%%%%%%%%%%%%%%%%%%%%%%%%%%%%%%%%%%%%%%%%%%%%%%%%%%%%%%% 

\section{Uniqueness of absolutely continuous
  stationary measure}
\label{sec:uniquen-absolut-cont}

As a consequence of the choice of adapted perturbations from
Theorem~\ref{nuero} and the family
$(\theta_\epsilon)_{\epsilon>0}$ of probability measures in
(\ref{eq:theta_ep}), we obtain the following.

\begin{theorem}\label{thm:uniquestationary}
  For each sufficiently small $\epsilon>0$ in the choice of
  $\zeta$ in the construction of an adapted random perturbation
  from (\ref{pertaditiva}) as in Theorem~\ref{nuero}, there
  exists a unique absolutely continuous and ergodic
  stationary measure for the random dynamical system
  $(f_{\underline{t}},\theta_\epsilon^{\mathbb{N}})$.
\end{theorem}

Consider the measure
$(f_{x})_{*}\theta_{\epsilon}^{\mathbb{N}}$ which is the
{\it push-foward} of the measure
$\theta_{\epsilon}^{\mathbb{N}}$ by $f_{t}:M\rightarrow M$
for a fixed $t\in\mbox{supp} \theta_{\epsilon}$, where we
write $f_{x}(\underline{t})$ for $f_{\underline{t}}(x)$.  We
first mention a simple way to ensure the existence of a
stationary measure for
$(f_{\underline{t}},\theta_\epsilon^{\mathbb{N}})$.

\begin{lemma}\label{le:existencestationary}
  For each sufficiently small $\epsilon>0$ in the choice of
  $\zeta$ in the construction of an adapted random
  perturbation from (\ref{pertaditiva}) as in
  Theorem~\ref{nuero} and for $x\in M$ fixed, each weak$^*$
  accumulation point of the sequence
  $\mu_{n}^{\epsilon}(x)
  =\frac{1}{n}\sum_{j=1}^{n}(f_{x}^{j})_{*}\theta_{\epsilon}^{\mathbb{N}}$
  is a stationary measure.
\end{lemma}

\begin{proof}
  Let $\mu^{\epsilon}$ be a weak$^*$ accumulation point of
  the sequence $(\mu_{n}^{\epsilon}(x))_{n}$. For each
  continuous $\phi: M\rightarrow\mathbb{R}$, we have by the
  Dominated Convergence Theorem
\begin{eqnarray}\label{eq:weak-conv}
\int\int\phi(f_{t}(y))d\mu^{\epsilon}(y)\theta_{\epsilon}(t)
&=&
\lim_{k\rightarrow
  +\infty}\int\int\phi(f_{t}(y))d\left(\frac{1}{n_{k}}\sum_{j=1}^{n_{k}}(f_{y}^{j})_{*}\theta_{\epsilon}^{\mathbb{N}}\right)d\theta_{\epsilon}(t)\nonumber
\\
&=&
\lim_{k\rightarrow+\infty}\frac{1}{n_{k}}\sum_{j=1}^{n_{k}}\int\int\phi(f_{t}(f_{\underline{t}}^{j}(x)))d\theta_{\epsilon}^{\mathbb{N}}(\underline{t})d\theta_{\epsilon}(t).
\end{eqnarray}
By definition of the perturbed iteration and of the infinite
product $\theta_\epsilon^{\mathbb{N}}$, and because
$\mu_{n_{k}}^{\epsilon}(x)\xrightarrow[n_{k}\to
+\infty]{}\mu^{\epsilon}$ in the weak$^{*}$ topology, the
limit in (\ref{eq:weak-conv}) equals
\begin{align*}
  \lim_{k\rightarrow
+\infty}\frac{1}{n_{k}}\sum_{j=1}^{n_{k}}
\int\phi(f_{\underline{t}}^{j+1}(x))\,
d\theta_{\epsilon}^{\mathbb{N}}(\underline{t})
=\int\phi \,d\mu^{\epsilon}.
\end{align*}
Hence
$\int\int
\phi(f_{t}(y))\,d\mu^{\epsilon}(y)d\theta_{\epsilon}(t)
=\int\phi \, d\mu^{\epsilon}$ and $\mu^{\epsilon}$ is a
stationary measure.
\end{proof}

%%%%%%%%%%%%%%%%%%%%%%%%%%%%%%%%%%%%%%%%%%%%%%%%%%%%%%%%%%%%

\subsection{Absolutely continuity and support with nonempty
  interior}
\label{sec:absolut-contin-suppo}

We now show that each stationary measure is absolutely
continuous with respect to Lebesgue measure $\lambda$ (a
volume form) in $M$.

\begin{proposition}\label{pr:abscontpullback} We have
  $(f_{x})_{*}\theta_{\epsilon}^{\mathbb{N}}<<\lambda$ for
  all $x\in M$.
\end{proposition}

We recall that from~\ref{rmk:thadapt} we have that $H$ is
never zero on $M$, and so $\zeta(x)\neq0$ for all $x\in M$.

\begin{proof}In fact, consider $A\subset M$ some ball in $M$
  which (we assume is a parallelizable manifold, e.g. an
  interval, the circle or a $n$-torus). We have
$$
\begin{array}{cll}
(f_{x})_{*}\theta_{\epsilon}^{\mathbb{N}}(A)
&=&
\theta_{\epsilon}^{\mathbb{N}}\{ \underline{t}: 
f_{\underline{t}}(x)\in A\}
\\
&=&
\theta_{\epsilon}^{\mathbb{N}}\{\underline{t}; f(x)+t_1
\cdot\zeta(x)\in A\}
\\
&=&
\theta_{\epsilon}\Big\{ t_1; t_1\in
\frac{A-f(x)}{\zeta(x)}\Big\}
\\
&=&
\frac1{\lambda(B_{\epsilon}(0))}\cdot
\lambda\left(\frac{A-f(x)}{\zeta(x)}\cap B_{\epsilon}(0)\right)
\\
&=&
\frac{1}{\zeta(x)}
\cdot
\frac1{\lambda(B_{\epsilon}(0))}
\cdot
\lambda\left((A-f(x))\cap B_{\epsilon}(0)\right)
% &=&\displaystyle\frac{1}{\xi(x)}\frac{\lambda(A)}{\lambda(B_{\epsilon}(0))}.
  \end{array}
$$
which shows that, if $\lambda(A)=0$, then
$(f_x)_*(\theta_\epsilon^{\mathbb N})(A)=0$.
\end{proof}

We observe that $B_\epsilon(0)\ni t\mapsto f_{t}(x)\in M$ is
continuous for each fixed $x\in M$. We also note that, since
the space $C^{0}(M,\mathbb{R})$ of continuous functions is
dense in the space $L^1(\mu^\epsilon)$ of Borel integrable
functions with respect to $\mu^\epsilon$, with the
$L^1$-norm, then the stationary condition in
Definition~\ref{eq:defstationary} holds also for any
$\mu$-integrable $\phi: M\rightarrow\mathbb{R}$.

\begin{lemma} Every stationary probability measure
  $\mu^{\epsilon}$ is absolutely continuous with respect to
  $\lambda$.
\end{lemma}

\begin{proof}
  From the above observation that the relation in Definition
  \ref{eq:defstationary} is true for all integrable
  functions, we have that for any Borel measurable subset
  $B\subset M$
\begin{align*}
  \mu^{\epsilon}(B) 
  = 
  \int\chi_{B}\,d\mu^{\epsilon}
  =
  \int\int \chi_{B}\circ f_{t}(y)\,d\mu^{\epsilon}(y)d\theta_{\epsilon}(t)
  =
  \int(f_{y})_{*}\theta^{\mathbb{N}}_{\epsilon}(B)\,d\mu^{\epsilon}(y)
\end{align*}
and if $\lambda(B)=0$, then we obtain $\mu^\epsilon(B)=0$
from Proposition~\ref{pr:abscontpullback}.
\end{proof}

From this we are able to show that the support of any
stationary measure has non-empty interior. Let
$\mu^{\epsilon}$ be a stationary measure and let us write
$S=\mathrm{supp}(\mu^{\epsilon})$.  Using again that the
relation in Definition~\ref{eq:defstationary} holds for
$\mu^\epsilon$-integrable functions
\begin{align*}
1
=
\int \chi_{S}(y)\,d\mu^{\epsilon}(y)
&=
\int\int \chi_{S}(f_{t}(y))\,
d\mu^{\epsilon}(y) d\theta_{\epsilon}(t)
\\
&=
\int\int\chi_{S}(f_{t}(y))\, 
d\theta_{\epsilon}(t)d\mu^{\epsilon}(y)
% \\
% &=&
% \int\int\chi_{\mathit{S}_{\epsilon}}\circ f^{j}_{\underline{t}}(y)\,
% d\theta_{\epsilon}^{\mathbb{N}}(\underline{t})d\mu^{\epsilon}(y).
\end{align*}
we conclude (since $0\le\chi_{S}\le1$) that
$\int\chi_{S}(f_{t}(y))\, d\theta_{\epsilon}(t)=1$ for
$\mu^\epsilon$-a.e. $y$. Therefore we get
$\chi_{S}(f_{t}(y))=1$, that is, $f_{t}(y)\in S$ for
$\theta_\epsilon$-a.e. $t$ and $\mu^\epsilon$-a.e. $y$.

In particular, $f_{t}(y)\in S$ for $t$ is a dense subset $D$
of $B_\epsilon(0)=\mathrm{supp}(\theta_\epsilon)$ by
definition of $\theta_\epsilon$.  In addition, since
$B_\epsilon(0)\ni t\mapsto f_{t}(y)\in M$ is continuous, we
also have $f_y(D)$ is dense in $f_y(B_\epsilon(0))$ and so
the closed set $S$ contains $B_{\zeta(y)}(f(y))$, the
closure of $f_y(D)$.  We obtain that \emph{$f_t(y)\in S$ for
  all $t\in B_\epsilon(0)$ and $\mu^\epsilon$-a.e. $y$.}

From the definition of $f_t(y)$ in (\ref{pertaditiva}), we
see that the image of $f_y(B_\epsilon(0))$ is the ball
around $f(y)$ with radius $\zeta(y)\neq0$.  Hence $S$ has
non-empty interior, as claimed.

%%%%%%%%%%%%%%%%%%%%%%%%%%%%%%%%%%%%%%%%%%%%%%%%%%%

\subsection{Every stationary measure is ergodic with full
  support}
\label{sec:every-station-measur}

Now we use that the  unperturbed transformation $f$ has a
dense orbit. Let $\mu^\epsilon$ be a stationary probability
measure. We have already shown that the support $S$ of
$\mu^\epsilon$ has non-empty interior and that $S$ is
\emph{almost invariant}. 

\begin{lemma}
  \label{le:fullinv}
  Let $(f_{\underline{t}},\theta_\epsilon^{\mathbb N})$ be a
  random dynamical system such that the unperturbed map
  $f=f_0$ is a local diffeomorphism outside a
  $\lambda$-measure zero set, has a dense positive orbit
  and the parameter $0$ belongs to the support of
  $\theta_\epsilon$.
  Then $\mu^\epsilon$ has full support:
  $S=\mathrm{supp}(\mu^\epsilon)=M$.
\end{lemma}

\begin{proof}
  Let $S_0\subset S$ be such that $\mu^\epsilon(S\setminus
  S_0)=0$ and $f_t(S_0)\subset S$ for all $t\in
  B_\epsilon(0)$ -- this was proved in the previous
  subsection. Hence we also have $\lambda(S\setminus S_0)=0$
  and so $\overline{S_0}=S$.

  We have that $f$ is locally a diffeomorphism outside a critical
  set $\mathcal{C}$ with $\lambda$-measure zero. Then
  $\lambda(f(S\setminus S_0))=0$ and, because
  $f(S)\setminus f(S_0)\subset f(S\setminus S_0)$, we get
  $\lambda(f(S)\setminus f(S_0))=0$.

  Thus
  $f(S)=f(\overline{S_0})\subseteq\overline{f(S_0)}\subseteq\overline
  S=S$, and \emph{$S$ is a positively $f$-invariant subset}.

We also know that the interior of $S$ is non-empty. Let
$w\in M$ have a positive dense $f$-orbit. Then there exists
$n>1$ such that $f^n(w)$ interior to $S$ and so
$M=\omega_f(x)\subset \overline{S}=S\subset M$.
\end{proof}

To show ergodicity of any stationary measure, we need some
known auxiliary results already obtained for maps with
hyperbolic times for random orbits, as stated below.

The first result gives properties of random hyperbolic times
similar to those of Lemma~\ref{lema 5.2}.

\begin{proposition}[Proposition 2.6 and Corollary 2.7 in
  \cite{AA03}]\label{pr:randhyptimes}
  There exist $\delta_{1},C_1>0$ such that, if $n$ is a
  $(\sigma,\delta)$-hyperbolic time for
  $(\underline{t},x)\in \Omega\times M$, then there exists a
  neighborhood $V_{n}(\underline{t},x)$ of $x$ in $M$ such
  that:
\begin{enumerate}
\item $f_{\underline{t}}^{n}$ maps $V_{n}(\underline{t},x)$
  diffeomorphically onto the ball of radius $\delta_{1}$
  centered at $f_{\underline{t}}^{n}(x)$;
\item
  $d(f_{\underline{t}}^{n-k}(y),f_{\underline{t}}^{n-k}(z))
  \leq
  \sigma^{k/2}\cdot d(f_{\underline{t}}^{n}(y),f_{\underline{t}}^{n}(z))$
  for all $1\leq k\leq n$ and $y,z\in
  V_{k}(\underline{t},x)$;
\item $C_1^{-1}\leq\frac{|\det
    Df_{\underline{t}}^{n}(y)|}{|\det
    Df_{\underline{t}}^{n}(z)|}\leq C_{1}$ for all $y,z\in
  V_{n}(\underline{t},x)$.
\end{enumerate}
\end{proposition}

The next result says that every non-trivial positively
invariant subset for random non-uniformly expanding
dynamical system must contain a ball of a definite size.

\begin{definition}[Random positively invariant set]
  We say that a subset $A\subset M$ is random positively
  invariant if, for $\mu^{\epsilon}$-almost every $x\in A$,
  we have that $f_{t}(x)\in A$ for
  $\theta_{\epsilon}$-almost every t.
\end{definition}

We note that if $A$ is random positively invariant and
$\lambda(A)>0$, then the closure of its Lebesgue density
points $A^+$ is also random positively invariant, since $A$
is dense in $A^+$.

\begin{proposition}[Proposition 2.13 in \cite{AlV13}]\label{pr:nucleus}
  For $\delta_{1}$ given by previous proposition, given any
  random positively invariant set $A\subset M$ with
  $\mu^{\epsilon}(A)>0$, there is a ball of radius
  $\delta_{1}/4$ such that $\lambda(B\setminus A^+)=0$.
\end{proposition}

The following is well-known from the theory of Markov chains.

\begin{lemma}[Lemma 8.2 in
  \cite{vdaraujo2000}] \label{le:restrict} The normalized
  restriction of a stationary measure to a random positively
  invariant set is a stationary measure.
\end{lemma}

Now we can prove that each stationary probability measure
$\mu^\epsilon$ for our random dynamical systems is ergodic.
Arguing by contradiction, let us assume that
$\mu^{\epsilon}$ is not ergodic.

Hence, there are random (positively) invariant sets $S_{1}$
and $S_{2}= M\setminus S_{1}$ such that both have
$\mu^{\epsilon}$-positive measure. From
Proposition~\ref{pr:nucleus} both sets contain a
$\delta_1/4$-ball. Thus there exist $n_1,n_2>1$ such that
$f^{n_1}(w)\in S_1$ and $f^{n_2}(w)\in S_2$, where $w$ is a
point with dense positive $f$-orbit. Therefore,
$\overline{S_1}=M=\overline{S_2}$ which is a contradiction.

% \begin{remark}
% Since $\mathit{S}_{\epsilon}=M$, by Birkhoof Theorem $\ref{birkhoff}$, we have
% $\mu^{\epsilon}(\mathcal{B}(\mu^{\epsilon}))>0$ and by absolute continuity of $\mu^{\epsilon}$, give us
% $\lambda(\mathcal{B}(\mu^{\epsilon}))>0$ and so we can say that $\mu^{\epsilon}$ is also physic measure.
% \end{remark}

%%%%%%%%%%%%%%%%%%%%%%%NOVA%%%%%%%%%%%%%%%%%%%%%%%%%%%%%%%%%%

\section{Stochastic stability}

Now we combine the results of the previous sections to prove
our main Theorem~\ref{principal}. We use the same strategy
as \cite{AA03} taking advantage of the uniformity of the
first hyperbolic time with respect to the adapted random
perturbations. Indeed, from the previous constructions and
from Remark~\ref{rmk:hbdd}, we
have that there exist $0<\sigma,\delta<1$ such that
\begin{align*}
  \hat h:\Omega\times M\to M, \quad (\underline t,x)\mapsto
  \inf\{k\ge1: k \text{  is a  }
  (\sigma,\delta)-\text{hyperbolic time for  } (\underline t,x)\}
\end{align*}
satisfies $\hat h(\underline t,x)=\hat h(\un 0,x)=h(x)\le
H(x)$ for all $\underline
t\in\supp\theta_\epsilon^{\mathbb{N}}$ for
$\lambda$-a.e. $x\in M$, where $\underline 0$ is the
constant sequence equal to zero and $h(x)$ denotes the first
hyperbolic time map associated to the unperturbed dynamics
of $f$, as defined in
Section~\ref{sec:adapted-random-pertu}.

Hence, if we assume that $h\in L^p(\lambda)$ for some $p>3$,
then we have also that the series
 \begin{equation}\label{c.unif}
 \|\hat h\|_1
 =
 \int \hat h\,d(\theta_\epsilon^{\mathbb{N}}\times \lambda)
 =
 \sum_{k=0}^{\infty}k\cdot
 (\theta_\epsilon^{\mathbb{N}}\times \lambda)
 \big(\{(\underline t,x)\colon \hat h(\underline t,x)=k\}\,\big)
 \end{equation}
 has {\em uniform $L^1$-tail}, that is, the series in the
 right hand side of \eqref{c.unif} converges uniformly to
 $\|\hat h\|_1$ (as a series of functions of the variable
 $\epsilon$).

 \begin{remark}\label{rmk:L1enough}
   For this argument it is enough that we assume $h\in
   L^1(\lambda)$, as long as $\hat h(\cdot,x)=h(x)$ for
   $\lambda$-a.e.  $x\in M$ is established.
 \end{remark}

 Now we can follow the same arguments as in \cite[Section
 5]{AA03}. We sketch them here for the convenience of the
 reader.  Since there exists a unique ergodic absolutely
 continuous stationary measure $\mu^\epsilon$ for all small
 enough $\epsilon>0$, we have that
\begin{align*}
  \mu_n^\epsilon =
  \frac1n\sum_{j=0}^{n-1}\int (f^j_{\underline
    t})_{*} \lambda\,
  d\theta_\epsilon^{\mathbb{N}}(\underline t).
\end{align*}
converges in the weak$^*$ topology to $\mu^\epsilon$ as
$n\to+\infty$. We define for each $\underline t\in
\Omega^{\mathbb{N}}$ and $n\geq 1$
\begin{align*}
  H_n(\un t)
  &=
  \{ x\in B(\mu^\epsilon)\colon \mbox{ $n$ is a
    $(\sigma,\delta)$-hyperbolic time for $(\un t,x)$ }\},
  \quad\text{and}
\\
  H^*_n(\un t)
  &=
  \{ x\in B(\mu^\epsilon)\colon \mbox{ $n$ is the first
 $(\sigma,\delta)$-hyperbolic time for $(\un t,x)$ }\}.
\end{align*}
Here $H^*_n(\un t)$ is the set of points $x$ for which $\hat
h(\un t,x)=n$.  For $n,k\geq 1$ we define $R_{n,k}(\un t)$
as the set of points $x$ for which $n$ is a
$(\sigma,\delta)$-hyperbolic time and $n+k$ is the first
$(\sigma,\delta)$-hyperbolic time after $n$, that is
\begin{align*}
  R_{n,k}(\un t)= \left\{x\in H_n(\un t)\colon \:f^n_{\un
      t}(x)\in H^*_k(\sigma^n\un t)\: \right\},
\end{align*}
where $\sigma:\Omega\circlearrowleft$ is the left shift map.
Now using the measures
\begin{align*}
 \nu^\epsilon_n
 &=
 \int ( f_{\un t}^n)_*\big(\lambda\mid H_n(\un
 t)\big)\,d\theta_\epsilon^\N(\un t)
\quad\text{and}\quad
\eta_n^\epsilon
=
\sum_{k=2}^\infty\sum_{j=1}^{k-1}\int (f_{\un
t}^{n+j})_*\big(\lambda\mid R_{n,k}(\un t)\big) 
\,d\theta_\epsilon^\N(\un t),
\end{align*}
we obtain the bound 
 $ \mu_n^\epsilon\leq
\frac{1}{n}\sum_{j=0}^{n-1}(\nu_j^\epsilon+\eta_j^\epsilon).$
The bounded distortion property of hyperbolic times provides
the following.

\begin{proposition}\cite[Proposition 5.2]{AA03}\label{pr:dens1}
   There is a constant $C_2>0$ such that for every $n\geq 0$
 and $\un t\in\Omega$ we have
$   \frac{d}{d\lambda}(f_{\un t}^n)_*\big(\lambda\mid H_n(\un
   t)\big)\leq C_2.$
\end{proposition}
Hence we have $\frac{d\nu_n^\epsilon}{d\lambda}\leq C_2$ for
every $n\geq 0$ and small $\epsilon>0$. We now control the
density of the measures $\eta_n^\epsilon$ so that we ensure
the absolute continuity of the weak$^*$ accumulation point
of $\mu^\epsilon$ when $\epsilon\searrow0$.
\begin{proposition}\cite[Proposition 5.3]{AA03}\label{pr:dens2}
  Given $\zeta>0$, there is $C_3(\zeta)>0$ such that for
  every $n\geq 0$ and $\epsilon>0$ we may bound
  $\eta_n^{\epsilon}$ by the sum of two measures
  $\eta_n^{\epsilon} \leq
  \omega^{\:\epsilon}+\rho^{\:\epsilon}$ satisfying
  $\frac{d\omega^{\:\epsilon}}{d\lambda}\leq C_3(\zeta)$ and
  $\rho^{\:\epsilon}(M)<\zeta.$
\end{proposition}
It follows from Propositions \ref{pr:dens1} and
\ref{pr:dens2} that the weak$^*$ accumulation points $\mu^0$
of $\mu^\ep$ when $\ep\searrow0$ cannot have singular part,
and so are absolutely continuous with respect to
$\lambda$. Moreover, from Remark~\ref{re:accinvariant} we
have that the weak$^*$ accumulation points $\mu^0$ of a
family of stationary measures are always $f$-invariant
measures.

From the properties of non-uniformly expanding maps stated
in Theorem~\ref{thm:abv}, we conclude that $\mu^0$ is
a convex linear combination of finitely many physical
measures of $f$. This proves stochastic stability under
adapted random perturbations.

In our setting, where $f$ is transitive, we
have a unique physical measure $\mu$ for $f$, thus
$\mu^0=\mu$.

%%%%%%%%%%%%%%%%%%%%%%%%%%BIBLIO%%%%%%%%%%%%%%%%%%%%%%%%%%%%

\def\cprime{$'$}

% \bibliographystyle{abbrv}
% \bibliography{../../../bibliobase/bibliography}

\end{document}